\documentclass[12pt,twoside]{article}
\usepackage{amsmath,amsfonts,amssymb,amsthm}
\usepackage{bbm}
\usepackage{graphicx}
\usepackage{hyperref}
\usepackage{color}
\newcommand{\grad}{\nabla}

\newcommand{\laplace}{\Delta}
\newcommand{\Nu}{\mbox{\it Nu}}
\renewcommand{\Re}{\mbox{\it Re}\,}
\newcommand{\Ra}{\mbox{\it Ra}}
\newcommand{\del}{\partial}
\renewcommand{\div}{\grad\cdot}
\newcommand{\N}{\mathbbm{N}}
\newcommand{\F}{{\mathcal F}}
\newcommand{\R}{\mathbbm{R}}

\newcommand{\Z}{\mathbbm{Z}}

\newcommand{\dz}{\partial_z}

\newcommand{\la}{\langle}
\newcommand{\ra}{\rangle}
\newcommand{\sign}{\mbox{\rm sign} }

\newtheorem{prop}{Proposition}
\newtheorem{theorem}{Theorem}
\newtheorem{lemma}{Lemma}

\newcommand{\tacka}{\, \cdot\,}

\begin{document}

\title{Laminar boundary layers in convective heat transport}

\author{Christian Seis\footnote{Department of Mathematics, University of Toronto, 40 St.\ George Street, M5S 2E4, Toronto, Ontario, Canada}}

\maketitle

\abstract{We study Rayleigh--B\'enard convection in the high-Rayleigh-number and high-Prandtl-number regime, i.e., we consider a fluid in a container that is exposed to strong heating of the bottom and cooling of the top plate in the absence of inertia effects. While the dynamics in the bulk are characterized by a chaotic convective heat flow, the boundary layers at the horizontal container plates are essentially conducting and thus the fluid is motionless. Consequently, the average temperature exhibits a linear profile in the boundary layers.

\medskip

In this article, we rigorously investigate the average temperature and oscillations in the boundary layer via local bounds on the temperature field. Moreover, we deduce that the temperature profile is indeed essentially linear close to the horizontal container plates. Our results are uniform in the system parameters (e.g.\ the Rayleigh number) up to logarithmic correction terms. An important tool in our analysis is a new Hardy-type estimate for the convecting velocity field, which can be used to control the fluid motion in the layer. The bounds on the temperature field are derived with the help of local maximal regularity estimates for convection-diffusion equations.}

\section{Introduction}
\subsection{Motivation}
Rayleigh--B\'enard convection is the flow of an incompressible Newtonian fluid in a container with an imposed temperature gradient due to heating of the bottom and cooling of the top plate. It is one of the classical models of fluid dynamics. With its numerous applications in geophysics, oceanography, meteorology, astrophysics and engineering, Rayleigh--B\'enard convection played a central role in experimental and theoretical physics since the turn of the last century. Depending on the heating rate, experimentalists and numerical analysts observe a wide range of flow pattern: from purely conducting (i.e., motionless) states, over steady and oscillatory fluid motions, to chaotic pattern and fully developed turbulence. In view of the richness of the observed phenomena, it is not surprising that over many years, Rayleigh--B\'enard convection has become a paradigm for nonlinear dynamics, including instabilities, bifurcations and pattern formation, e.g.\ \cite{Chandrasekhar, DrazinReid04, Getling98}. We refer to \cite{Siggia94, Kadanoff, BodenschatzPeschAhlers, AhlersGrossmannLohse09} for reviews and further references. 

\medskip

We consider Rayleigh--B\'enard convection in the so-called high-Rayleigh-number and high-Prandtl-number regimes. This means, the applied temperature forcing is so strong that the fluid motion is unorganized. Moreover, inertia effects are negligible, so that, in particular, the observed flow pattern is rather {\em chaotic} than turbulent. Thinking of a container that is infinitely extended in horizontal directions, boundary effects on the vertical sidewalls of the container become irrelevant. In this situation, the flow pattern shows a clear spatial separation of the relevant heat transfer mechanisms: thin laminar boundary layers in the vicinity of the horizontal plates, in which heat is essentially conducted and a large bulk that is characterized by a convective heat flow. Since near the rigid walls of the container, the fluid is almost at rest, the relevant heat transport mechanism in the boundary layers is conduction, and thus, the temperature profile is essentially linear. Fluctuations around this profile take place on large length scales, with a small amplitude. It is in these boundary layers where the majority of the temperature drop between the hot bottom and the cold top boundaries happens. In the bulk, far from the rigid container walls, convection is only limited by viscous friction. The temperature is essentially equilibrated around the mean temperature of the system. However, due to the strong temperature forcing, overload heat is produced on the bottom, which generates instabilities of the boundary layer: The boundary layer bursts and hot fluid parcels with warm tails, so-called plumes, detach and flow into the much colder bulk due to buoyancy forces. While rising, the plumes push aside the above fluid and are itself in turn deflected. In this way, the plumes take the forms of stalks with caps on their tops --- the plumes obtain their characteristic mushroom-like shapes, cf.\ \cite{ZocchiMosesLibchaber90, Kadanoff}. (Of course, the corresponding effects can also be observed on the cold top plate. For convenience, we restrict our considerations to the bottom boundary.)

\medskip

The rigorous understanding of the dynamic flow pattern in the chaotic regime is far from being satisfactory. Until today, most of the work is devoted to the scaling of the Nusselt number, that is, the scaling of the average upward heat flux as a function of the imposed temperature forcing \cite{DC96, CD99, DC01, dor06, Wang08, OS11}. In this paper, we focus on the boundary layers at the horizontal bottom and top plates of the container. So far, boundary layer theories in Rayleigh--B\'enard convection were primarily derived as intermediate steps towards the understanding of the Nusselt number scaling. Since due to the absence of fluid motion, the boundary layers offer the main resistance for the heat flow through the container, the Nusselt number is dominantly determined by the boundary layer, and thus, for scaling theories, the width of these boundary layer is of particular interest. (In fact, the Nusselt number is inversely proportional to the width of the layer.) However, because of their extreme thinness, experimental and numerical investigations  of the boundary layers are very difficult, and still, no generally accepted boundary layer theory is available for Rayleigh--B\'enard convection. An overview on the current state of experimental and theoretical research for more general (than the one considered in the present paper) Rayleigh--B\'enard experiments can be found in \cite[Sect.\ VI]{AhlersGrossmannLohse09}.

\medskip

The present work is a first attempt to partially characterize the observed pattern in chaotic high-Rayleigh-number convection --- at least in the phenomenologically most regular region. Our aim is a rigorous justification of the laminar profile and thus the dominant role of conduction in the boundary layer. We rigorously establish local bounds on the temperature field and its gradients in the boundary layers that are (up to logarithmic corrections) uniform in the system parameters (Theorem \ref{AT1}). These bounds indicate that the temperature field close to the horizontal plates is indeed essentially laminar and fluctuations only happen on relatively large length scales, with a weak dependence on the heating rate. In other words, our analysis proves that heat is transported essentially via conduction. Moreover, we can deduce that the average temperature decays linearly and with slope equal to the Nusselt number (Theorem \ref{AT1bis}), as it is expected in conducting boundary layers.


\medskip

The remainder of the paper is organized as follows: In Subsection \ref{S:Model}, we introduce the mathematical model and the Nusselt number; in Subsection \ref{S:Results}, we present our main results and discuss the method of this paper. Section \ref{C3S4} is devoted to the analysis of the velocity field. Finally, Section \ref{S:Proof} contains the proofs of Theorem \ref{AT1} and \ref{AT1bis}.

\subsection{Model and Nusselt number}\label{S:Model}

Despite the complexity of the observed phenomena, the mathematical model for Rayleigh--B\'enard convection is relatively simple. If density variations of the fluid are sufficiently small and the thermal diffusivity is negligible compared to kinematic viscosity, the problem can be modelled by the infinite-Prandtl-number limit of the Boussinesq equations:
\begin{eqnarray}
\partial_t T + u\cdot \grad T - \laplace T &=& 0,\label{Aeq1}\\
\div u&=& 0,\label{Aeq2}\\
-\laplace u + \grad p&=& Te.\label{Aeq3}
\end{eqnarray}
Here, $T$ is temperature, $u$ the fluid velocity, $p$ the hydrodynamic pressure, and $e$ the upward unit vector. We suppose that the container has the simple form $[0,\Lambda)^{d-1}\times[0,H]$, where we refer to the first $d-1$ coordinates as the horizontal ones, and to the last coordinate as the vertical one. We complete the system with periodic boundary conditions in the horizontal directions; at the rigid top and bottom plates, we set
\begin{equation}\label{eq:9}
T = \left\{\begin{array}{l} 1\mbox{ for } z=0\\0\mbox{ for }z=H\end{array}\right\}\quad\mbox{and}\quad u=0 \mbox{ for }z\in\{0,H\}.
\end{equation}
Thus, we suppose uniform heating/cooling of the bottom/top plate, and no-slip boundary conditions for the fluid velocity. We follow the convention to write the space coordinate and the velocity field as $x=(y,z)\in\R^{d-1}\times\R$ and $u=(v,w)\in\R^{d-1}\times\R$, respectively.

\medskip

The system is nondimensionalized and admits a single control parameter: the container height $H$. The side length $\Lambda$ of the period cell is chosen arbitrary and has no significance in the subsequent investigation. We distinguish two regimes: 1) the {\em linear regime} for small container heights, $H\ll1$. Here, heat transfer between bottom and top plate is exclusively due to conduction and the temperature field stabilizes in a linear profile. 2) the {\em chaotic regime} for large container heights, $H\gg1$. In this regime, the fluid experiences a strong temperature forcing, which leads to the formation of chaotic flow pattern, as described in the previous subsection. Notice that the present nondimensionalization differs from the common one in the physics literature, where the control parameter that measures the applied temperature forcing is the Rayleigh number $\Ra$. Both, $H$ and $\Ra$ are related via $H=\Ra^{1/3}$, which explains the term ``High-Rayleigh-number convection'' for the chaotic regime. For the sake of completeness, we like to mention that between these two regimes, there is a critical container height $H_*\sim1$ for which the linear profile becomes unstable and a steady circular fluid (convection rolls) flow sets in. These convection rolls become itself unstable for larger heights and a cascade of bifurcations can be observed, until the chaotic behavior sets in for $H\gg1$. 

\medskip

The efficiency of the heat transport is measured by the Nusselt number $\Nu$, which is defined as the average upward heat flux,
\[
\Nu\;:=\; \frac1H\int_0^H\la wT - \partial_z T\ra\, dz\;\stackrel{\eqref{eq:9}}{=}\; \frac1H\int_0^H \la wT\ra\, dz + \frac1H.
\]
Here, the brackets $\la\tacka\ra$ denote the horizontal space and time average, i.e.,
\[
\la f\ra\;:=\; \limsup_{t_0\uparrow\infty}\frac1{t_0} \int_0^{t_0} \frac1{\Lambda^{d-1}}\int_{[0,\Lambda)^{d-1}} f(t,y)\, dydt
\]
for any function $f(t,y)$. Some equivalent expressions for the Nusselt number can be derived:
\begin{eqnarray}
\Nu &=& \la wT - \del_z T\ra\quad\mbox{for all }z\in[0,H]\label{AN5}\\
&=& \int_0^H \la |\grad T|^2\ra\, dz\label{AN3}\\
&=& \frac1H\int_0^H \la|\grad u|^2\ra\, dz + \frac1H \label{AN4}.
\end{eqnarray}
Indeed, for \eqref{AN5}, we average \eqref{Aeq1} in horizontal space and time and apply the boundary conditions \eqref{eq:9}. Identity \eqref{AN3} can be derived by testing the heat equation \eqref{Aeq1} with $T - (1-z/H)$, integrating by parts and using the divergence-free condition \eqref{Aeq2} and the boundary conditions \eqref{eq:9}, while \eqref{AN4} follows from testing \eqref{Aeq3} with $u$, integrating by parts and using \eqref{Aeq2} and \eqref{eq:9} again.

\medskip

Over many years, the dependence of the Nusselt number on the control parameter $H$ has been studied extensively, cf.\ \cite{GrossmannLohse00, AhlersGrossmannLohse09} and references therein. In the absence of inertia, as in the present model, the expected scaling of $\Nu$ is the following: In the linear regime $H\ll1$, when heat transport is essentially due to conduction, the Nusselt number scales like the imposed temperature gradient: $\Nu\sim 1/H$. 
In the chaotic regime $H\gg1$, it is conjectured that conductive and convective heat transport are balanced in the system, in the sense that $\Nu\sim1$.
While the first scaling law can be easily established, cf.\ \cite[Theorem 2]{S}, a rigorous derivation of the second one remains reluctant until today. The presently best upper bound on the Nusselt number in the chaotic regime was derived by Otto and the author in \cite{OS11} and is optimal only up to a double-logarithmic factor:
\begin{equation}\label{110}
\Nu\;\lesssim\; \ln^{1/3}(\ln H)\quad\mbox{as }H\gg1.
\end{equation}
Observe that proving rigorous lower bounds on $\Nu$ is substantially different from proving upper bounds: Lower bounds depend sensitively on the particular choice of the initial data. Indeed, there are ungeneric solutions to the system \eqref{Aeq1}--\eqref{eq:9} for which the heat flux is less efficient, for instance the purely conducting state $T=1- z/H$, $u=0$. In this case, it is $\Nu=1/H\ll1$. Therefore, one can only expect to prove (physically relevant) a-priori {\em upper} bounds on the Nusselt number in the chaotic regime.

\subsection{Main results of this paper}\label{S:Results}
From now on, we restrict our attention to the chaotic regime, that means, we assume that
\[
H\gg1.
\]
Moreover, in order to rule out ``ungeneric'' configurations like purely conducting solutions (as described in the last subsection) or steady fluid motions, e.g.\ convection rolls, we focus on solutions for which
\[
\Nu\;\gtrsim\; 1\quad\mbox{as }H\gg1.
\]

\medskip

We are interested in statistical properties of the temperature field in the boundary layer in the chaotic regime. Our first result gives upper bounds on  derivatives of the temperature field up to fourth order. We show that in a boundary layer of order one, the temperature can be controlled in appropriate Sobolev norms {\em uniformly in terms of $H$ --- modulo logarithmic prefactors}.

\begin{theorem}\label{AT1}Assume that $H\gg1$ and $\Nu\gtrsim1$. Let $\alpha\in\N$, $\alpha\le 4$. Then there exists a $\gamma>0$ such that
\begin{equation}\label{AT1.2}
\int_0^{1} \la |\grad^{\alpha} T|^{4/\alpha} \ra\, dz
\;\lesssim\; \ln^{\gamma} H.
\end{equation}
\end{theorem}

In the above statement, the size of the boundary layer is not explicitly fixed to one.  Instead, we can rather think of these estimates to hold in {\em any} boundary layer whose width is of {\em order one} or {\em logarithmic in $H$}. However, for simplicity we restrict ourselves in the following to the order-one case. Likewise, for symmetry reasons, the same bounds hold in the upper boundary layer. Moreover, for the benefit of a concise statement and to trim the estimates in the proofs, we do not compute explicit exponents $\gamma$ on the logarithms. Upper bounds on $\gamma$ (at least for the second and third order derivatives), can be found in the author's PhD thesis \cite{S}.

\medskip 
The bounds stated in Theorem \ref{AT1} can be considered in the context of rigorous upper bounds on the Nusselt number initiated by Constantin and Doering. Indeed, in view of \eqref{AN3}, the bounds in \cite{CD99,  dor06, OS11}, e.g.\ \eqref{110}, can be read as
\[
 \int_0^H \la|\grad T|^2\ra\, dz\;\lesssim\; \ln^{\gamma} H
\]
for some $\gamma>0$, and thus, estimate \eqref{AT1.2} appears to be a new contribution in the study of universal bounds on the temperature field in infinite-Prandtl-number Rayleigh--B\'enard convection.

\medskip

The estimates in Theorem \ref{AT1} can be used to show that, in the vertical boundary layers, the temperature profile is essentially linear:

\begin{theorem}\label{AT1bis}
Assume that $H\gg1$ and $\Nu\gtrsim1$. There exists a $\gamma>0$ such that for any $z\in[0,1]$
\[
\left|\la \left. T\right|_{z}\ra - \left( 1 - z\Nu\right)\right|\;\lesssim\;z^3\ln^{\gamma} H.
\]
\end{theorem}

It follows from the above Theorem that
\[
\la\left. T\right|_{z}\ra\;\approx\;  1 - z\Nu \quad\mbox{for }z\ll \ln^{-\gamma/3}H.
\]
Again, by symmetry, an analogue statement can be derived in the upper boundary layer. As the majority of the temperature drop between the bottom and the top plate in the Rayleigh--B\'enard experiment just happens in the boundary layers, the heat flow must be inversely proportional to the width of the layer. The precise slope can be easily deduced from \eqref{AN5} by choosing $z=0$ and exploiting the boundary conditions \eqref{eq:9} for $w$, namely $\left.\partial_z\right|_{z=0} \la T\ra=-\Nu$. Now, in Theorem \ref{AT1bis}, we recover this slope all over the boundary layer. The main benefit of the above estimate is the cubic control of the deviation in vertical direction around the linear profile, which only depends logarithmically and thus {\em weakly} on $H$.

\medskip

The method we use  in this paper in order to derive the bounds on the temperature field stated in Theorem \ref{AT1} are motivated by bounds derived by Otto for the viscous Burgers equation in \cite{Otto09} and Otto and Ramos for the Navier--Stokes equation in \cite{OttoRamos}. In both papers, the authors establish $L_p$ maximal-regularity-type estimates with tools borrowed from Harmonic Analysis. Using and refining the techniques from these two papers, we derive new {\em local} maximal regularity estimates for convection-diffusion equations, see Propositions \ref{AP2} and \ref{AP3} on p.\ \pageref{AP2}f. After differentiating the heat equation \eqref{Aeq1} in spatial directions and and localizing in the boundary layer, these estimates can be applied to gain control on gradients of the temperature field. We have to comment on the choice of the norms in the assertions of Theorem \ref{AT1}. These are determined by the leading order error terms, which come from differentiating and localizing the transport nonlinearity. Only few tools are known that are applicable in order to bound these terms uniformly in $H$ (modulo logarithms): localized $L_2$ bounds on the velocity field in the spirit of those derived in \cite{dor06}, cf.\ Proposition \ref{AP1} on page \pageref{AP1}, the maximum principle for $T$, and interpolation inequalities of the form
\[
\int |\grad^{\alpha-1} T|^{4/(\alpha-1)}\, dx\;\lesssim\; \sup |T|^{4/\alpha(\alpha-1)} \int |\grad^{\alpha} T|^{4/\alpha}\, dx.
\]
The value of $p$ in the $L_p$ maximal-regularity estimates is chosen optimally with respect to the estimates at hand.

\section{Bounds on the velocity field}\label{C3S4}

In this section, we collect a series of $L_2$ bounds on the velocity fields: some known results from \cite{dor06} and \cite{OS11} and some new results together with their proofs. The main new result, stated in Proposition \ref{AP1} below, is an estimate on the shear velocity $v$ in the boundary layer. As this estimate might be of independent interest, in this section, we state all results for boundary layers of {\em arbitrary} width $\delta$.

\medskip

We consider the stationary Stokes equation for the velocity field $u=(v,w)\in \R^{d-1}\times \R$ and the hydrodynamic pressure $p$ in the Rayleigh--B\'enard problem, i.e., equations \eqref{Aeq2} \& \eqref{Aeq3} equipped with periodic boundary conditions in all $d-1$ horizontal directions and with no-slip conditions on the vertical boundaries, i.e., $w=v=0$ for $z\in\{0,H\}$. We observe that, because of \eqref{Aeq2}, we have the additional information that $\partial_z w =-\grad_y\cdot v=0$ for $z\in\{0,H\}$. It is convenient to eliminate the pressure term in \eqref{Aeq3} with the help of the incompressibility condition \eqref{Aeq2}. This leads to a fourth-order equation for $w$:
\begin{equation}\label{Aeq4}
\laplace^2w\; =\; -\laplace_y T\quad\mbox{and}\quad w=\dz w=0 \mbox{ for }z\in\{0,H\}.
\end{equation}
Here, $\laplace_y$ denotes the Laplace operator in the horizontal variables. Indeed, applying the divergence operator to \eqref{Aeq3} yields $\laplace p= \partial_z T$, and therefore, applying the Laplace operator to the vertical component of \eqref{Aeq3}, we obtain $\laplace^2 w = \laplace \partial_z p - \laplace T= -\laplace_y T$. In a similar way, making additionally use of the boundary conditions, we see that $v$ satisfies the equation $\laplace_y v = \grad_y \grad_y\cdot v$, and thus by \eqref{Aeq2}, $v$ is determined by
\begin{equation}\label{AR3}
v= (-\laplace_y)^{-1}\grad_y\del_z w.
\end{equation}

\medskip

Furthermore, since there are no external forces acting on the fluid velocity, we have 
\[
\la u(t_0)\ra\; =\;0
\]
for every $t_0>0$, which follows immediately from averaging \eqref{Aeq2} and \eqref{Aeq3} in horizontal direction, and exploiting the boundary conditions for $u$.

\medskip

At this point, we like to recall some bounds on the vertical velocity component $w$, derived in \cite{dor06} and \cite{OS11}. The key estimate of \cite{dor06} (here, slightly extended by including the horizontal gradient $|\grad_y w|^2$) is the following: Let $w$ and $T$ be periodic in $y$ and satisfy \eqref{Aeq4}. Then
\begin{equation}\label{AR1}
 \int_0^H \frac{\la |\grad w|^2\ra}z\, dz\; \lesssim\; \int_0^H \frac{\la Tw\ra}{z}\, dz.
\end{equation}
In Lemma \ref{AL1} below, we will extend this result to the horizontal velocity component $v$. Notice that, as a byproduct of the result in \cite{dor06}, the r.h.s.\ of \eqref{AR1} is controlled by the Nusselt number (modulo a logarithmic correction):
\begin{equation}\label{AR2}
\int_0^H \frac{\la Tw\ra}{z}\, dz\;\lesssim\; (\ln H)\Nu.
\end{equation}
We also refer to Lemma \cite[Lemma 2]{OS11} for a direct proof of this estimate. Consequently, \eqref{AR1} \& \eqref{AR2} imply
\begin{lemma}[\cite{dor06, OS11}]\label{AL4}
\begin{equation}\label{AL4.1}
\int_0^H \frac{\la|\grad w|^2\ra}z\, dz\; \lesssim\; (\ln H)\Nu.
\end{equation}
\end{lemma}
In \cite[Lemma 3, see also eq.\ (18)]{OS11}, the authors derive an $L_2$ maximal regularity estimate for the third-order derivatives of $w$ from \eqref{Aeq4}. By the Nusselt number representations \eqref{AN3}\&\eqref{AN4}, the result can be stated as follows:
\begin{equation}\label{5}
\int_0^H \la|\grad_y^{-1} \grad^4 w|^2\ra\, dz\;\lesssim\; \Nu,
\end{equation}
or, invoking the formula \eqref{AR3}:
\begin{lemma}[\cite{OS11}]\label{AL6}
\begin{equation}\label{AL6.1}
\int_0^H \la |\grad^3 u|^2\ra\, dz\; \lesssim\; \Nu.
\end{equation}
\end{lemma}
The $\dot H^{-1}_2$ norm in \eqref{5} can best be understood on the Fourier level: For any periodic function $f(y)$ we define
\[
\frac1{\Lambda^{d-1}}\int_{[0,\Lambda)^{d-1}} |\grad_y^{-1} f|^2\, dy\; :=\; \sum_{k\in \frac{2\pi}{\Lambda}\Z^{d-1}} |k|^{-2}|\F f(k)|^2.
\]
Here $\F f(k)$ is the Fourier transform of $f$ at wave number $k$, i.e.,
\begin{equation}\label{OSF}
\F f(k) \;:=\; \frac1{\Lambda^{d-1}}\int_{[0,\Lambda)^{d-1}} e^{ik\cdot y} f(y)\, dy.
\end{equation}

\medskip

Our first new result is the following $\dot H_2^{-1}$ bound on $\partial_zw$, which extends estimate \eqref{AR1} to the horizontal velocity component $v$ via the relation \eqref{AR3}:

\begin{lemma}\label{AL1}
Let $w$ and $T$ be periodic in $y$ and satisfy \eqref{Aeq4}. Then for any $0\le\delta\ll H$ and $\alpha>1$ it holds
\begin{equation}\label{AL1.6}
\int_0^{\delta} \frac{\la |\grad_y^{-1} w|^2\ra}{z^{5-\alpha}} \, dz
+\int_0^{\delta} \frac{\la |\grad_y^{-1}\dz w|^2\ra}{z^{3-\alpha}} \, dz
 \; \lesssim\;\frac{\delta^{\alpha}}{\alpha-1}\ln^2 \left(\frac{H}{\delta}\right) \int_0^H \frac{\la  T w \ra}z\, dz,
\end{equation}
and
\begin{equation}\label{AL1.7}
 \int_{\delta}^H \frac{\la |\grad_y^{-1} w|^2\ra}{z^5} \, dz
 +\int_{\delta}^H \frac{\la |\grad_y^{-1}\dz w|^2\ra}{z^3} \, dz
 \; \lesssim\;\ln^2 \left(\frac{H}{\delta}\right) \int_0^H \frac{\la  T w \ra}z\, dz.
\end{equation}
\end{lemma}

As we shall see, the main ingredients for estimates \eqref{AL1.6} and \eqref{AL1.7} are Hardy-type inequalities for the function $\frac{w}{z^2}$, i.e., inequalities of the form
 \begin{equation}\label{AL1.4}
\int_0^H \frac1{z^{1+\nu}}\phi^2\, dz\;\lesssim\; \int_0^H z^{1-\nu} (\frac{d}{dz} \phi)^2\, dz
\end{equation}
for some $\nu\in \R$. This inequality fails for $\nu=0$ --- but only logarithmically. We exploit this logarithmic failure in two different ways. In the ``bulk'' $(\delta,H)$, the logarithmic failure of \eqref{AL1.4} with $\nu=0$ produces the prefactor $\ln^2(H/\delta)$ in \eqref{AL1.7}. Notice that the estimates leading to \eqref{AL1.7} are sharp. In particular, the logarithmic prefactor can be recovered in the critical Hardy inequality when choosing $\phi\approx \ln (z/\delta)$ (with the appropriate boundary conditions at $z=H$). In the ``boundary layer'' $(0,\delta)$, we prove a subcritical Hardy inequality, i.e., \eqref{AL1.4} with some $\nu = -\alpha<0$, which leads to \eqref{AL1.6}. Notice that in view of the prescribed boundary conditions for $w$, it holds $\frac{\la|\grad_y^{-1} w|^2\ra}{z^5} \sim \frac{\la|\grad_y^{-1}\dz w|^2\ra}{z^3}\sim\frac1z$ for $z\ll1$, which just beats the integrability. Consequently, we can expect \eqref{AL1.4} to be true only for $\nu<0$, i.e., $\alpha>
 0$. In light of this observation, our bound for the boundary layer term seems to be suboptimal, since it requires $\alpha >1$. However, it is optimal in terms of scaling --- a fact that will be exploited in Proposition \ref{AP1} below. The logarithmic prefactor in \eqref{AL1.6} stems --- roughly speaking --- from the fact that the boundary terms at $z=\delta$ have to be estimated via the bulk estimate \eqref{AL1.7}.

\medskip

\begin{proof}[Proof of Lemma \ref{AL1}]
We prove the result on the Fourier level, cf.\ \eqref{OSF}. Under Fourier transformation, \eqref{Aeq4} translates into
\[
|k|^2 \F w - 2 \frac{d^2}{dz^2} \F w + \frac1{ |k|^2} \frac{d^4}{dz^4} \F w\; =\; \F  T
\]
and
\[
\F w=\frac{d}{dz} \F w = 0\mbox{ for }z\in\{0,H\}.
\]

\medskip

We start by recalling the following inequality that has been derived in \cite[p.\ 238\&239]{dor06} and that can be easily reproduced by testing the equation for $\F w$ with $z^{-1} \overline{\F w}$ and integrating by parts: With $\phi = \frac{\F w}{z^2}$ we have
\begin{equation}\label{1}
\frac1{|k|^2} \int_0^H z^3 |\frac{d^2}{dz^2} \phi|^2\, dz
\;\le\;
\int_0^H \frac1z \F  T \overline{ \F w}\, dz.
\end{equation}
Notice that we may assume --- applying an approximation argument --- that $\phi$ has the same boundary values as $\F w$, so that boundary terms vanish when integrating by parts. Now the statements in \eqref{AL1.6} and \eqref{AL1.7} follow from the three Hardy(-type) estimates
\begin{eqnarray}
\int_0^H z|\frac{d}{dz} \phi|^2\, dz& \le& \int_0^H z^3|\frac{d^2}{dz^2} \phi|^2\, dz,\label{2}\\
\int_{\delta}^H \frac1z | \phi|^2\, dz &\lesssim& \ln^2\left(\frac{H}{\delta}\right) \int_{\delta}^H z|\frac{d}{dz}\phi|^2\, dz\label{3}
\end{eqnarray}
and
\begin{eqnarray}
\lefteqn{
\frac{\alpha-1}{\delta^{\alpha}} \left(\int_0^{\delta} \frac1{z^{5-\alpha}} |\F w|^2\, dz + \int_0^{\delta} \frac1{z^{3-\alpha}} |\frac{d}{dz} \F w|^2\, dz\right) + \int_{\delta}^H \frac1{z^3} |\frac{d}{dz} \F w|^2\, dz}\nonumber\\
&\lesssim& \int_{0}^H z|\frac{d}{dz} \phi|^2\, dz + \int_{\delta}^H \frac1z |\phi|^2\,dz\label{4}.\hspace{15em}
\end{eqnarray}
Indeed, for instance estimate \eqref{AL1.6} can be derived as follows:
\begin{eqnarray*}
\lefteqn{\frac{\alpha-1}{\delta^{\alpha}}\left(\frac1{|k|^2}\int_0^{\delta} \frac1{z^{5-\alpha}} |\F w|^2\, dz + \frac1{|k|^2}\int_0^{\delta} \frac1{z^{3-\alpha}}|\frac{d}{dz} \F w|^2\, dz\right)}\\
&\stackrel{\eqref{4}}{\lesssim}& \frac1{|k|^2}\int_{0}^H z|\frac{d}{dz}\phi|^2\, dz + \frac1{|k|^2}\int_{\delta}^H \frac1z|\phi|^2\, dz\\
&\stackrel{\eqref{3}, H\gg\delta }{\lesssim}& \ln^2\left(\frac{H}{\delta}\right) \frac1{|k|^2}\int_{0}^H z|\frac{d}{dz}\phi|^2\, dz\\
&\stackrel{\eqref{2}}{\lesssim}&\ln^2\left(\frac{H}{\delta}\right) \frac1{|k|^2}\int_0^H z^3|\frac{d^2}{dz^2}\phi|^2\, dz\\
&\stackrel{\eqref{1}}{\lesssim}&\ln^2\left(\frac{H}{\delta}\right) \int_0^H \frac1z \overline{\F w}\F T\, dz.
\end{eqnarray*}
In view of the Plancherel Theorem, summing over all wave numbers $k$ and averaging in time yields \eqref{AL1.6}. Estimate \eqref{AL1.7} is derived similarly.

\medskip

It remains to prove statements \eqref{2}--\eqref{4}. Observe that \eqref{2} is a standard Hardy estimate. We omit its proof which is very similar to the one of the critical Hardy inequality \eqref{3}. The latter follows from
\[
\int_1^{\hat H} \frac1{\hat z}|\hat \phi|^2\, d\hat z
\;\lesssim\;
(\ln^2 \hat H)\int_{1}^{\hat H} \hat z|\frac{d}{d\hat z} \hat\phi|^2\, d\hat z
\]
via the transformation $z=\delta \hat z$, $H=\delta \hat H$, and $\phi(z) = \hat \phi(\hat z)$. This estimate is easily established. We integrate by parts and estimate with the help of the Cauchy--Schwarz inequality:
\begin{eqnarray*}
 \int_1^{\hat H} \frac1{\hat z}| \hat \phi|^2\, d\hat z
&=& \int_1^{\hat H} \frac{d}{d \hat z} (\ln \hat z) |\hat \phi|^2\, d\hat z\\
&\le&  2\int_1^{\hat H} (\ln \hat z) |\hat\phi||\frac{d}{d\hat z} \hat\phi|\, d \hat z\\
&\le&2 (\ln \hat H)\left(\int_1^{\hat H}  \frac1{\hat z}|\hat\phi|^2\, d\hat z\int_1^{\hat H} \hat z|\frac{d}{d\hat z}\hat \phi|^2\, d\hat z\right)^{1/2},
\end{eqnarray*}
which immediately yields \eqref{3}.

\medskip

The proof of \eqref{4} is slightly more involved. We introduce the weight function
\[
\rho(z) =\left\{\begin{array}{ll}\frac{z^{1+\alpha}}{\delta^{\alpha}}&\mbox{for }0\le z\le \delta,\\
z&\mbox{for }\delta \le z \le H,
\end{array}\right.
\]
and have the obvious estimate
\begin{equation}\label{AL1.1bis}
\int_0^H \rho(z) |\frac{d}{dz}\phi|^2\, dz
\;\le\;
 \int_0^H z |\frac{d}{dz}\phi|^2\, dz,
\end{equation}
because of $\rho(z)\le z$. Varying the weight is helpful, since it allows to treat both boundary layer and bulk estimates simultaneously, and thus avoids the appearance of boundary terms at $z=\delta$ (for which we do not know an appropriate estimate). Using the definition of $\phi$, we easily compute
\begin{eqnarray*}
\int_0^H \rho(z) |\frac{d}{dz}\phi|^2\, dz &=& \int_0^H \rho(z) \frac1{z^4}|\frac{d}{dz} \F w|^2\, dz- 4 \Re \int_0^H \rho(z) \frac1{z^5} \F w \overline{\frac{d}{dz} \F w}\, dz\\
&&\mbox{}  + 4 \int_0^H \rho(z) 
\frac1{z^6}|\F w|^2\, dz,
\end{eqnarray*}
where $\Re\, Z$ denotes the real part of a complex number $Z$. We see via integration by parts that
\begin{eqnarray*}
\lefteqn{\frac{\alpha}{\delta^{\alpha}}\int_0^{\delta} \frac1{z^{5-\alpha}} |\F w|^2\, dz}\\
 &=& \int_0^H \frac{d}{dz}\left(\frac{\rho(z)}z \right) \frac1{z^4} |\F w|^2\, dz\\
&=& 4 \int_0^H \rho(z) \frac1{z^6} |\F w|^2 \, dz-2 \Re\int_0^H \rho(z) \frac1{z^5} \F w \overline{\frac{d}{dz} \F w}\, dz,
\end{eqnarray*}
so that we may rewrite the above formula as
\begin{eqnarray}
\lefteqn{\int_0^H \rho(z) |\frac{d}{dz}\phi|^2\, dz}\nonumber\\
&=& \frac1{\delta^{\alpha}}\int_0^{\delta} \frac1{z^{3-\alpha}}|\frac{d}{dz} \F w|^2\, dz\nonumber\\
&&\mbox{}- \frac2{\delta^{\alpha}}\Re \int_0^{\delta} \frac1{z^{4-\alpha}} \F w \overline{\frac{d}{dz} \F w}\, dz  + \frac{\alpha}{\delta^{\alpha}}\int_0^{\delta} \frac1{z^{5-\alpha}}|\F w|^2\, dz\nonumber\\
&&\mbox{}+ \int_{\delta}^H \frac1{z^3}|\frac{d}{dz} \F w|^2\, dz - 2\Re\int_{\delta}^H \frac1{z^4} \F w \overline{\frac{d}{dz} \F w}\, dz.\label{B1}
\end{eqnarray}
We invoke the Cauchy-Schwarz and Young inequalities to deduce
\begin{eqnarray*}
\lefteqn{\int_0^H \rho(z) |\frac{d}{dz}\phi|^2\, dz}\\
&\ge& \frac{\alpha-1}{\delta^{\alpha}}\int_0^{\delta} \frac1{z^{5-\alpha}}|\F w|^2\, dz + \frac12\int_{\delta}^H \frac1{z^3}|\frac{d}{dz} \F w|^2\, dz - 2\int_{\delta}^H \frac1{z^5} |\F w|^2 \, dz,
\end{eqnarray*}
which becomes
\begin{eqnarray*}
\lefteqn{\frac{\alpha-1}{\delta^{\alpha}}\int_0^{\delta} \frac1{z^{5-\alpha}}|\F w|^2\, dz + \int_{\delta}^H \frac1{z^3}|\frac{d}{dz} \F w|^2\, dz}\\
&\lesssim& \int_0^H z |\frac{d}{dz}\phi|^2\, dz + \int_{\delta}^H  \frac{1}{z}|\phi|^2\, dz
\end{eqnarray*}
in view of the definition of $\phi$ and formula \eqref{AL1.1bis}. Now, we obtain control on the second term on the l.h.s.\ of \eqref{4} via \eqref{B1} and the Young inequality. This completes the proof of Lemma \ref{AL1}.
\end{proof}

The following proposition is a consequence of the Lemmas \ref{AL4}, \ref{AL6} \& \ref{AL1}.

\begin{prop}\label{AP1}
For any $0\le\delta\le H$ it holds
\[
\int_0^{\delta} \la |\grad_y v|^2 \ra\, dz + \int_0^{\delta} \la |\grad w|^2 \ra\, dz\;
\lesssim\;
\delta(\ln H)\Nu.
\]
Moreover, if $\frac1H\ll\delta\ll\ln H$ then
\[
\int_0^{\delta} \la |\partial_z v|^2\ra\, dz
\;\lesssim\;\delta(\ln^3 H)\Nu.
\]
\end{prop}

As an immediate consequence of the previous proposition together with the Nusselt number bound \eqref{110}, we observe for later citations that, in a boundary layer of {\em order one}, we have the estimates
\begin{equation}\label{Aeq:3}
\int_0^1 \la|u|^2\ra\, dz\;\le\; \int_0^{1} \la |\grad u|^2\ra\, dz
\;\lesssim\;\ln^{\gamma} H,
\end{equation}
for some $\gamma>0$. Here the first estimate is due to Poincar\'e's inequality.

\begin{proof}[Proof of Proposition \ref{AP1}]
The bounds on $\grad w$ and $\grad_y v$, follow directly from Lemma \ref{AL4} via
\begin{equation}\nonumber
\int_0^{\delta} \la |\grad w|^2\ra\, dz\; \le\; \delta \int_0^{H}\frac{ \la |\grad w|^2\ra}z\, dz\; \stackrel{\eqref{AL4.1}}{\lesssim}\; \delta(\ln H)\Nu
\end{equation}
and
\[
\int_0^{\delta} \la |\grad_y v|^2 \ra\, dz
\; \stackrel{\eqref{AR3}}{\le}\;
\int_0^{\delta} \la |\grad_y^{-1}\grad_y\partial_zw|^2 \ra\, dz
\;=\;
\int_0^{\delta} \la (\partial_zw)^2 \ra\, dz.
\]
It remains to estimate the $\partial_z v$-term. Because of \eqref{AR3}, we have to show
\begin{equation}\label{AP1.0bis}
\int_0^{\delta} \la |\grad_y^{-1} \dz^2 w|^2 \ra\, dz\;
\lesssim\;\delta(\ln^3 H)\Nu.
\end{equation}
First, we apply the elementary estimate
\[
\int_0^{\delta} |\frac{d}{dz} f|^2\, dz
\;\lesssim\;
\frac1{\delta^2}\int_0^{\delta} |f|^2\, dz + \delta^4 \int_0^{\delta} |\frac{d^3}{dz^3} f|^2\, dz
\]
to $f= \grad_y^{-1}\del_z w$ and obtain
\[
\int_0^{\delta} \la |\grad_y^{-1} \dz^2 w|^2 \ra\, dz
\;\lesssim\;
\frac1{\delta^2}\int_0^{\delta} \la |\grad_y^{-1} \dz w|^2 \ra\, dz
+ \delta^4 \int_0^{\delta} \la |\grad_y^{-1} \dz^4 w|^2 \ra\, dz.
\]
For the first term on the r.\ h.\ s.\ of the above equation, we invoke the boundary layer estimate from Lemma \ref{AL1} with $\alpha=3$,
\begin{eqnarray*}
\int_0^{\delta} \la |\grad_y^{-1}\dz w|^2\ra \, dz
&\stackrel{\eqref{AL1.6}}{\lesssim}&
\delta^3\ln^2 \left(\frac{H}{\delta}\right) \int_0^H \frac{\la T w \ra}z\, dz\\
&\stackrel{\eqref{AR2}}{\lesssim}& \delta^3\ln^2 \left(\frac{H}{\delta}\right) (\ln H)\Nu\\
&\stackrel{\frac1H\ll \delta}{\lesssim}&
\delta^3 (\ln^3 H)\Nu.
\end{eqnarray*}
The second term is estimated with the help of Lemma 3 in \cite{OS11}:
\begin{equation}\nonumber
\int_0^H \la|\grad_y^{-1}\del_z^4 w|^2\ra\, dz
\;\stackrel{\eqref{5}}{\lesssim}\; \Nu.
\end{equation}
We combine the above estimates and deduce \eqref{AP1.0bis} since $\delta\ll\ln H$.
\end{proof}

We conclude this section with a global bound on the velocity field, which combines classical Calderon--Zygmund theory with the maximum principle for the temperature.

\begin{lemma}\label{AL3}
For any $1<q<\infty$, it holds
\begin{equation}\label{AL3.1}
\frac1{H^q}\int_0^H \la|\grad u|^q\ra\, dz\;\lesssim\; \int_0^H \la |\grad^2 u|^q\ra\, dz
\;\lesssim\;
H.
\end{equation}
\end{lemma}

\begin{proof}[Proof of Lemma \ref{AL3}] The first estimate in \eqref{AL3.1} follows immediately from the Poincar\'e inequality and the boundary conditions \eqref{eq:9}. The second statement follows from $L_q$ maximal regularity estimates for the Stokes equation,
\begin{equation}\label{AL3.2}
\int_0^H \la |\grad^2 u|^q\ra\, dz
\;\lesssim\; \int_0^H \la |T|^q\ra\, dz,
\end{equation}
and the maximum principle for the temperature in the sense that
\[
\la \sup_x |T|^q\ra\; \le\; 1,
\]
cf.\ \eqref{AT1.4} below and discussion on page \pageref{AT1.4}. Although \eqref{AL3.2} is a very classical estimate, we have to argue that the estimate is uniform in the aspect ratio $\frac{H}{\Lambda}$. By rescaling $x$, $u$, and $p$, and redefining $\Lambda$, we may w.l.o.g.\ assume that $H=1$. Furthermore, since the pressure is unique up to a constant, we suppose that
\begin{equation}\label{AL3.2bis}
\int_0^1\int_{(0,2)^{d-1}}p\, dydz\; =\; 0.
\end{equation}
Our starting point is the analogue result for the Stokes equation on an infinite strip. Consider
\begin{eqnarray}
-\laplace \tilde u + \grad\tilde p&=& f,\label{AL3.3}\\
\div \tilde u &=& g\label{AL3.4}
\end{eqnarray}
in $\R^{d-1}\times (0,1)$ with non-slip boundary conditions at the vertical boundary: $\tilde u=0$ for $z\in\{0,1\}$. Then we have
\begin{eqnarray}
\lefteqn{\int_0^1\int_{\R^{d-1}}|\grad^2 \tilde u|^q\, dydz
 + \int_0^1\int_{\R^{d-1}}|\grad \tilde p|^q\, dydz}\nonumber\\
&\lesssim&
\int_0^1\int_{\R^{d-1}}|f|^q\, dydz 
+\int_0^1\int_{\R^{d-1}}|\grad g|^q\, dydz \label{AL3.5}
\end{eqnarray}
by classical theory, cf.\ \cite[Chapter IV]{Galdi}, as long as the r.h.s.\ is finite. In order to apply this estimate, we introduce a smooth cut-off function $\eta=\eta(y)$ with the properties that $\eta\in[0,1]$, $\eta=1$ for $|y|_{\infty}\le N\Lambda$, $\eta=0$ for $|y|_{\infty}\ge N\Lambda+1$, and $\sup_y |\grad_y \eta|\lesssim1$, $\sup_y |\grad_y^2 \eta|\lesssim1$. Here, $N\in\N$ is an arbitrary, but fixed, number that has to be chosen explicitly at the end of the proof. It is readily checked that
\[
\tilde u=\eta u, \quad \tilde p=\eta p, \quad f=\eta Te - \laplace_y \eta u - 2\grad_y\eta\cdot\grad_y u + {\grad_y \eta \choose 0} p, \quad g=  \grad_y\eta\cdot v
\]
solve \eqref{AL3.3}\&\eqref{AL3.4}. In particular, if we use the cut-off properties of $\eta$, we deduce form \eqref{AL3.5} the estimate
\begin{eqnarray*}
\lefteqn{\int_0^1 \int_{|y|_{\infty}\le N \Lambda} |\grad^2 u|^q\, dydz + \int_0^1 \int_{|y|_{\infty}\le N \Lambda} |\grad p|^q\, dydz}\\
&\lesssim&
\int_0^1\int_{|y|_{\infty}\le N \Lambda+1 } |T|^q\, dydz\\
&&\mbox{} + \int_0^1\int_{N \Lambda\le |y|_{\infty}\le N \Lambda+1}( |u |^q  + |\grad u|^q + |p|^q)\, dydz.
\end{eqnarray*}
In view of the periodicity of $T$, $u$, and $p$ in $y$, we may rewrite the above equation as
\begin{eqnarray*}
\lefteqn{N^{d-1}\int_0^1 \int_{[0,\Lambda)^{d-1}} |\grad^2 u|^q\, dydz + N^{d-1}\int_0^1 \int_{[0,\Lambda)^{d-1}} |\grad p|^q\, dydz}\\
&\lesssim&
N^{d-1}\int_0^1\int_{[0,\Lambda)^{d-1}} |T|^q\, dydz + \int_0^1\int_{(0,2)^{d-1}} |T|^q\, dydz\\
&&\mbox{} +
\int_0^1\int_{(0,2)^{d-1}}( |u |^q  + |\grad u|^q + |p|^q)\, dydz.
\end{eqnarray*}
Notice that, by the homogeneous boundary conditions of $u$ at $z\in\{0,H\}$, we can apply the Poincar\'e inequality in the vertical variable both for $u$ and $\grad u$ and obtain
\[
\int_0^1\int_{(0,2)^{d-1}} |u |^q\,  dydz
\;\lesssim\;  \int_0^1\int_{(0,2)^{d-1}} |\grad u |^q \, dydz
\;\lesssim\;  \int_0^1\int_{(0,2)^{d-1}} |\grad^2 u |^q \, dydz.
\]
Moreover, thanks to \eqref{AL3.2bis}, the Poincar\'e inequality in $y$ yields
\[
\int_0^1\int_{(0,2)^{d-1}} |p |^q\,  dydz
\;\lesssim\;
\int_0^1\int_{(0,2)^{d-1}} |\grad p |^q \, dydz.
\]
We conclude that
\begin{eqnarray*}
\lefteqn{N^{d-1}\int_0^1 \int_{[0,\Lambda)^{d-1}} |\grad^2 u|^q\, dydz + N^{d-1}\int_0^1 \int_{[0,\Lambda)^{d-1}} |\grad p|^q\, dydz}\\
&\lesssim&
N^{d-1}\int_0^1\int_{[0,\Lambda)^{d-1}} |T|^q\, dydz+\int_0^1\int_{(0,2)^{d-1}} |T|^q\, dydz\\
&&\mbox{} + 
\int_0^1\int_{(0,2)^{d-1}}|\grad^2 u |^q\, dydz
+\int_0^1\int_{(0,2)^{d-1}}|\grad p |^q\, dydz.
\end{eqnarray*}
Thus, choosing $N$ sufficiently large ($N\Lambda\gtrsim 1$ if $\Lambda\ll1$), we may absorb velocity and pressure terms into the l.\ h.\ s. Normalizing by the system size and averaging in time, this proves \eqref{AL3.2}.
\end{proof}

\section{Estimates on the temperature field}\label{S:Proof}

In this section, we prove Theorem \ref{AT1} and \ref{AT1bis}. We first show, how Theorem \ref{AT1} implies Theorem \ref{AT1bis}.

\begin{proof}[Proof of Theorem \ref{AT1bis}] The proof is elementary. Observe first, that
\begin{equation}\label{eq:6}
\left|\la \del^3_{z|z=0} T \ra\right|\; \lesssim\; \ln^{\gamma} H.
\end{equation}
Indeed, this is a direct consequence of the Poincar\'e and H\"older inequality and \eqref{AT1.2} with $\alpha=3$ and $\alpha=4$:
\begin{eqnarray*}
\left|\la \del^3_{z|z=0} T\ra\right| &\le& \int_0^1 \la |\del^3_z T| \ra\, dz + \int_0^1 \la |\del^3_{z|z=0} T - \del^3_z T|\ra\, dz\\
&\le& \left(\int_0^1 \la |\del^3_z T|^{4/3}\ra\, dz\right)^{3/4} +  \int_0^1 \la |\del_z^4 T|\ra\, dz\\
&\stackrel{\eqref{AT1.2}}{\lesssim}& \ln^{\gamma} H.
\end{eqnarray*}

\medskip

By Tailor expansion, we may write the average temperature as
\begin{eqnarray*}
\la  T_{|z}\ra &=& \la T_{|z=0}\ra + z \la \partial_{z|z=0}T\ra + \frac12 z^2  \la \partial_{z|z=0}^2 T\ra \\
&&\mbox{} + \frac16 z^3 \la \partial_{z|z=0}^3 T\ra+  \frac1{24} \int_0^z \la \partial_{z}^4 T\ra (z-\tilde z)^3\, d\tilde z.
\end{eqnarray*}
We first identify the expressions on the right up to the quadratic term with the help of the boundary conditions \eqref{eq:9}: We immediately see that $\la T_{|z=0}\ra=1$, the Nusselt number representation \eqref{AN5} evaluated at $z=0$ yields  $\la \partial_{z|z=0} T\ra=-\Nu$, and exploiting \eqref{eq:9} in \eqref{Aeq1} again, we see that $\la \partial_{z|z=0}^2 T\ra=0$. Thus, the above identity turns into
\[
\la  T_{|z}\ra - \left(1-z\Nu\right)\;=\; \frac16 z^3 \la \partial_{z|z=0}^3 T\ra+  \frac1{24} \int_0^z \la \partial_{z}^4 T\ra (z-\tilde z)^3\, d\tilde z.
\]
To bound the terms on the r.h.s., we invoke \eqref{eq:6} and estimate
\[
 \left|\int_0^z \la \partial_{z}^4 T\ra (z-\tilde z)^3\, d\tilde z\right|\;\lesssim\; z^3 \int_0^1 \la |\partial_{z}^4 T|\ra \, d z\;\stackrel{\eqref{AT1.2}}{\lesssim}\; z^3\ln^{\gamma} H.
\]
This proves Theorem \ref{AT1bis}.
\end{proof}

For the proof of Theorem \ref{AT1}, we need some preparations. We start with three interpolation inequalities, stated in the Lemmas \ref{L9} and \ref{L10}. 

\begin{lemma}\label{L9}
Let $\zeta$ denote a periodic function with vertical period $2H$. Then
\begin{equation}\label{L9.1}
\left(\int_{-H}^H \la |\grad \zeta|^{4/3}\ra\, dz\right)^{3/4}\; \lesssim\; \left(\int_{-H}^H\la \zeta^2\ra\, dz\right)^{1/4} \left(\int_{-H}^H\la |\grad^2 \zeta|\ra\, dz\right)^{1/2}.
\end{equation}
\end{lemma}

\begin{proof}[Proof of Lemma \ref{L9}] We first observe that it is enough to prove estimate \eqref{L9.1} in the case where $\zeta = \zeta(s)$ is a periodic function on some interval $[0,a]$, i.e.,
\[
\left(\int_0^a |\zeta'|^{4/3}\, ds\right)^{3/4}\;\lesssim\; \left(\int_0^a\zeta^2\, ds\right)^{1/4}\left(\int_0^a |\zeta''|\, ds\right)^{1/2}.
\]
The general case follows (roughly speaking) when choosing $s=x_i$, integrating over the remaining $d-1$ variables, and summing over all $i$. Moreover, the above statement is a consequence of the estimate
\begin{equation}\label{L9.2}
\left(\int_0^b |\zeta'|^{4/3}\, ds\right)^{3/4}\;\lesssim\;   \left(\int_0^b\zeta^2\, ds\right)^{1/2} + \int_0^b |\zeta''|\, ds\quad\mbox{for all $b>0$,}
\end{equation}
as one can see by rescaling length by $a/b$ and optimizing in $b$. We now turn to the proof of \eqref{L9.2}. If $b\ll1$, then \eqref{L9.2} reduces to a Poincar\'e-type inequality. Indeed, making use of the periodic boundary conditions of $\zeta$, we estimate
\[
\left(\int_0^b |\zeta'|^{4/3}\, ds\right)^{3/4}\;\le\; b^{3/4} \sup_s |\zeta'|\;\le\; b^{3/4} \int_0^b |\zeta''|\, ds\;\stackrel{b\ll 1}{\le}\; \int_0^b |\zeta''|\, ds.
\]
Observe that \eqref{L9.2} is a standard Ehrling estimate in the case where $b\sim 1$ --- regardless of the actual boundary conditions. Therefore, if $b\gg1$, we can invoke a summation argument based on the ``$b\sim 1$''-case. For instance, if $b=K\in\N$, we have
\begin{eqnarray*}
\left(\int_0^K |\zeta'|^{4/3}\, ds\right)^{3/4} &=& \sum_{k=1}^K \left(\int_{k-1}^k |\zeta'|^{4/3}\, ds\right)^{3/4}\\
&\lesssim& \sum_{k=1}^K\left( \left(\int_{k-1}^k\zeta^2\, ds\right)^{1/2} + \int_{k-1}^k |\zeta''|\, ds\right)\\
&=& \left(\int_0^K\zeta^2\, ds\right)^{1/2} + \int_0^K |\zeta''|\, ds.
\end{eqnarray*}
This proves Lemma \ref{L9}.
\end{proof}

The following estimates are designed to make use of the maximum principle for the temperature field $|T|\le1$, cf.\ \eqref{AT1.4} below.

\begin{lemma}\label{L10}
Let $\zeta$ denote a periodic function with vertical period $2H$. Then

\begin{equation}\label{L10.2}
\int_{-H}^H \la |\grad\zeta|^4\ra\, dz \;\lesssim\;\sup_x |\zeta|^2 \int_{-H}^H \la |\grad^2\zeta|^2\ra\, dz
\end{equation}
and
\begin{equation}\label{L10.1}
\int_{-H}^H \la |\grad^2 \zeta|^2\ra\, dz \;\lesssim\;\sup_x |\zeta|^{2/3} \int_{-H}^H \la  |\grad^3 \zeta|^{4/3}\ra\, dz.
\end{equation}
\end{lemma}

\begin{proof}[Proof of Lemma \ref{L10}.]
We start with estimate \eqref{L10.2} and fix $i\in\{1,\dots,d\}$. Integration by parts and applying the H\"older inequality yields
\begin{eqnarray*}
\int_{-H}^H \la (\del_i \zeta)^4\ra\, dz &=& -3\int_{-H}^H \la \zeta (\del_i\zeta)^2\del_i^2\zeta\ra\, dz\\
&\lesssim& \sup_x |\zeta|\left(\int_{-H}^H \la (\del_i\zeta)^4\ra\, dz\right)^{1/2}\left(\int_{-H}^H \la (\del_i^2\zeta)^2\ra\, dz\right)^{1/2}.
\end{eqnarray*}
It remains to apply the Young inequality and sum over all $i$ to deduce \eqref{L10.2}.

\medskip

Now, we turn to \eqref{L10.1}. Again, it is enough to show this statement for $\grad^2 \zeta$ replaced by $\partial_i^2\zeta$. The estimate for mixed second derivatives follows via integration by parts and the Cauchy--Schwarz inequality. The proof is very similar to the one of \eqref{L10.2}. We integrate by parts and use the H\"older inequality:
\begin{eqnarray*}
\int_{-H}^H \la (\del_i^2 \zeta)^2\ra\, dz& =& -\int_{-H}^H \la\del_i \zeta \del_i^3\zeta\ra\, dz\\
&\le& \left(\int_{-H}^H \la (\del_i\zeta)^4\ra\, dz\right)^{1/4}\left(\int_{-H}^H \la |\del_i^3 \zeta|^{4/3}\ra\, dz\right)^{3/4}.
\end{eqnarray*}
In order to obtain \eqref{L10.1}, we just insert \eqref{L10.2} (with $\grad^2\zeta$ replaced by $\del_i^2\zeta$ as in the derivation of \eqref{L10.2}) in the above estimate.
\end{proof}

To study the convection-diffusion equation \eqref{Aeq1}, it is convenient to rewrite this equation in terms of $\theta=T-(1-z/H)$, that is,
\begin{equation}\label{AT1.1}
\del_t \theta + u\cdot \grad \theta - \laplace \theta\; =\; \frac1H w.
\end{equation}
This is beneficial since $\theta$ has homogeneous boundary conditions: $\theta=0$ for $z\in\{0,H\}$. We may w.l.o.g.\ assume that
\begin{equation}\label{AT1.4}
\sup_x |T|\; =\; \sup_x |\theta|\;\le\;1.
\end{equation} 
In fact, by the maximum principle for the temperature, $T\in[0,1]$ (and thus $\theta \in [-1,1]$) is preserved during the evolution. Even if this condition is not satisfied initially, $T$ attains this temperature range exponentially fast. Since we are working with long time averages, finite time intervals in which \eqref{AT1.4} is violated will not affect our results.

\medskip

Notice that by the definition of $\theta$ and the Nusselt number representation \eqref{AN3}, we have
\begin{equation}\label{eq:7}
\int_0^{H}\la |\grad\theta|^2\ra\, dz
\;\sim\; \Nu.
\end{equation}
For higher order derivatives, it holds $\grad^2T=\grad^2\theta$, $\grad^3T=\grad^3\theta$ and so on.

\medskip

We observe that odd reflection of $\theta$ and $w$ and even reflection of $v$ leaves the equations \eqref{AT1.1} and \eqref{Aeq2} invariant. Regarding the boundary conditions of $\theta$, $w$, and $v$, we infer from \eqref{AT1.1} that $\del_z^2\theta=0$ for $z\in\{0,H\}$. Hence we may think from here on of $\theta$ as a (restriction of a) $2H$-periodic $C^3$ function in $z$, of $w$ as a (restriction of a) $2H$-periodic $C^1$ function in $z$, and of $v$ as a (restriction of a) $2H$-periodic $C^0$ vector field in $z$. In particular, $\theta$ is four times weakly differentiable, $w$ is twice weakly differentiable, and $v$ is one time weakly differentiable in $z$.

\medskip

The core of the proof of Theorem \ref{AT1} are the following maximal regularity-type estimates for convection-diffusion equations:
\begin{prop}\label{AP2}
Let $\zeta$, $u$, and $f$ be smooth periodic functions with vertical period $2H$ satisfying
\begin{eqnarray}
\del_t \zeta + u \cdot \grad \zeta - \laplace \zeta & =& f,\label{AP2.1}\\
\div u&=&0.\label{AP2.1bis}
\end{eqnarray}
Assume that $u=0$ for $z=0$. Assume further that
\begin{equation}\label{AP2.1cis}
\mbox{\rm supp}_z\zeta \subset [-\delta, \delta]
\end{equation}
for some $0<\delta\le H$. Let $1\le r< 2$. Then
\begin{eqnarray}
\lefteqn{\left(\int_{-H}^{H} \la |\grad^2 \zeta|^r\ra\, dz\right)^{1/r}}\nonumber\\
&\lesssim& M\left(\int_{-H}^H \la |f|^r\ra\, dz\right)^{1/r} \nonumber\\
&&\mbox{} + e^{-M} \left(\int_{-H}^H \la |\grad f|^r\ra\, dz\right)^{1/r} \nonumber\\
&&\mbox{}+ \delta^{(2-r)/2r}\left(\int_{-H}^H \la \zeta^2\ra\, dz\right)^{1/2} \nonumber\\
&&\mbox{} + N \left(\int_{-2\delta}^{2\delta} \la |\grad  u|^2\ra\, d z \right)^{1/2}
\left(\int_{-H}^{H} \la  |\zeta|^{2r/(2-r)}\ra\, d z\right)^{(2-r)/2r}\nonumber\\
&&\mbox{}+ e^{-N} \left(\int_{-H}^H \la|\grad u|^4\ra\, dz\right)^{1/4}\left(\int_{-H}^H \la |\grad \zeta|^{4r/(4-r)}\ra\, dz\right)^{(4-r)/4r}\label{AP2.0dis}
\end{eqnarray}
for any $ M,N\in\N_0$.
\end{prop}

In the case $r=2$ we can do better:

\begin{prop}\label{AP3}
Let $\zeta$, $u$, and $f$ be smooth periodic functions with vertical period $2H$ satisfying
\begin{eqnarray}
\del_t \zeta + u \cdot \grad \zeta - \laplace \zeta & =& f,\label{AP3.1}\\
\div u&=&0.\label{AP3.1bis}
\end{eqnarray}
Assume that
\begin{equation}\label{AP3.1cis}
\mbox{\rm supp}_z\zeta \subset [-\delta, \delta]
\end{equation}
for some $0<\delta\le H$. Then
\begin{equation}\label{AP3.2}
\int_{-H}^H \la |\grad^2 \zeta|^2\ra\, dz\;\lesssim\; \int_{-H}^H \la f^2\ra\, dz + \sup_x |\zeta| \int_{-\delta}^{\delta} \la |\grad u|^2\ra\, dz.
\end{equation}
\end{prop}

In the proof of Theorem \ref{AT1}, we will apply Propositions \ref{AP2} and \ref{AP3} to the localized temperature field and its derivatives, i.e., $\zeta=\eta\theta,\, \partial_i(\eta\theta),\, \partial_{ij}(\eta\theta)$, where $\eta=\eta(z)$ denotes a cut-off function. Before, we like to comment on \eqref{AP2.0dis} and \eqref{AP3.2}. The derivation of a maximal regularity estimate for convection-diffusion equations in $L_2$ is a straight-forward computation, cf.\ proof of Proposition \ref{AP3}. Its $L_r$ analogue, however, is challenging. The difficulty in deriving scale-invariant $L_r$ estimates stems from the fact that, for our purpose, the convection term $u\cdot\grad \zeta$ can not be treated as a error term: In fact, local (and thus, up to logarithms uniform) bounds on the velocity field are known only in $L_2$, see the previous section. Hence ordinary maximal regularity theory in Sobolev spaces is not suitable for our purposes. On the level of Besov norms, however, the transport-nonlinearity can be controlled in an elegant way. At first glance, we expect an estimate of the type
\begin{eqnarray}
\lefteqn{\left(\int_{-H}^H \la |\grad^2\zeta|^{r}\ra\, dz\right)^{1/r}}\nonumber\\
&\lesssim&
\left(\int_{-H}^H \la |f|^{r}\ra\, dz\right)^{1/r}\nonumber\\
&&\mbox{}+ \left(\int_{-H}^H \la |\grad u|^{2}\ra\, dz\right)^{1/2}\left(\int_{-H}^H \la |\zeta|^{2r/(2-r)}\ra\, dz\right)^{(2-r)/2r}.\nonumber
\end{eqnarray}
Compared to this estimate, the additional error terms that occur in \eqref{AP2.0dis} are caused the fact that Sobolev and Besov norms are in general not equivalent. These error terms produce an additional logarithmic prefactor in Theorem \ref{AT1}.

\medskip

We postpone the proof of Propositions \ref{AP2} and \ref{AP3} to the end of this paper and start with a series of global bounds on the temperature field, all based on Proposition \ref{AP1}:

\begin{lemma}\label{AL9}
There exists a $\gamma>0$ such that
\begin{eqnarray}
\int_0^H\la|\grad T|^4\ra\, dz
&\lesssim&
\int_0^H \la|\grad^2 T|^2\ra\, dz\nonumber\\
&\lesssim&\int_0^{H} \la |\grad^3 T|^{4/3}\ra\, dz\nonumber\\
&\lesssim&\int_0^{H} \la |\grad^4 T|\ra\, dz\nonumber\\
&\lesssim&
(\ln^{\gamma} H)H.\label{AL9.2}
\end{eqnarray}
\end{lemma}

\begin{proof}[Proof of Lemma \ref{AL9}.] We first remark that it is enough to prove the statement with $T$ replaced by $\theta$. Observe that in view of the Lemmas \ref{L9} and \ref{L10} above, which we apply to $\zeta=\grad^2 \theta$ and $\zeta=\theta$, respectively, and the maximum principle for the temperature \eqref{AT1.4}, it is
\begin{eqnarray}
\int_{-H}^{H} \la |\grad \theta|^4\ra\, dz&\lesssim& \int_{-H}^{H} \la |\grad^2 \theta|^2\ra\, dz, \label{108}\\
\int_{-H}^{H} \la |\grad^2 \theta|^2\ra\, dz &\lesssim& \int_{-H}^{H} \la |\grad^3 \theta|^{4/3}\ra\, dz,\label{107}\\
\int_{-H}^{H} \la |\grad^3 \theta|^{4/3}\ra\, dz &\lesssim& \int_{-H}^{H} \la |\grad^4 \theta|\ra\, dz.\label{109}
\end{eqnarray}
Indeed, \eqref{108} and \eqref{107} follow immediately from \eqref{L10.2} and \eqref{L10.1}, respectively, with $\zeta=\theta$ via \eqref{AT1.4}. For \eqref{109}, we apply \eqref{L9.1} with $\zeta=\grad^2 \theta$ and estimate:
\begin{eqnarray*}
\int_{-H}^{H} \la |\grad^3 \theta|^{4/3}\ra\, dz &\stackrel{\eqref{L9.1}}{\lesssim}& 
\left(\int_{-H}^{H} \la |\grad^2 \theta|^{2}\ra\, dz\right)^{1/3}\left(\int_{-H}^{H} \la |\grad^4 \theta|\ra\, dz\right)^{2/3}\\
&\stackrel{\eqref{107}}{\lesssim}&\left(\int_{-H}^{H} \la |\grad^3 \theta|^{4/3}\ra\, dz\right)^{1/3}\left(\int_{-H}^{H} \la |\grad^4 \theta|\ra\, dz\right)^{2/3}.
\end{eqnarray*}
It remains to invoke Young's inequality to obtain \eqref{109}.

\medskip

As a consequence of \eqref{108}--\eqref{109}, for \eqref{AL9.2} it is enough to prove
\begin{equation}\label{114}
\int_0^H \la |\grad^4 \theta|\ra\, dz\; \lesssim\; (\ln^{\gamma} H)H.
\end{equation}
We differentiate \eqref{AT1.1} w.r.t.\ $x_i$ and $x_j$ for $i,j\in\{1,\dots, d\}$, that is
\[
\partial_t \del_{ij}\theta + u\cdot \grad\del_{ij} \theta - \laplace \del_{ij}\theta\;=\; \frac1H\del_{ij} w - \partial_i u\cdot\grad\partial_j\theta - \partial_j u\cdot\grad\partial_i\theta - \del_{ij} u\cdot \grad\theta.
\]
Applying Proposition \ref{AP2} with $r=1$, and $\delta=H$ to $\zeta=\del_{ij}\theta$ and $f= \frac1H\del_{ij} w - \partial_i u\cdot\grad\partial_j\theta - \partial_j u\cdot\grad\partial_i\theta - \del_{ij} u\cdot \grad\theta$ and summing over all $1\le i,j\le d $ yields
\begin{eqnarray}
\lefteqn{\int_{-H}^{H} \la |\grad^4 \theta|\ra\, dz}\nonumber\\
&\lesssim& M\int_{-H}^H \la |f|\ra\, dz \nonumber\\
&&\mbox{} + e^{-M} \int_{-H}^H \la |\grad f|\ra\, dz \nonumber\\
&&\mbox{}+ H^{1/2}\left(\int_{-H}^H \la |\grad^2\theta|^2\ra\, dz\right)^{1/2} \nonumber\\
&&\mbox{} + N \left(\int_{-H}^{H} \la |\grad  u|^2\ra\, d z \right)^{1/2}
\left(\int_{-H}^{H} \la  |\grad^2\theta|^2\ra\, d z\right)^{1/2}\nonumber\\
&&\mbox{}+ e^{-N} \left(\int_{-H}^H \la|\grad u|^4\ra\, dz\right)^{1/4}\left(\int_{-H}^H \la |\grad^3\theta|^{4/3}\ra\, dz\right)^{3/4}\nonumber.
\end{eqnarray}
We invoke \eqref{107} and \eqref{109} successively and estimate with the help of Young's inequality
\begin{eqnarray}
\int_{-H}^{H} \la |\grad^4 \theta|\ra\, dz
&\lesssim& M\int_{-H}^H \la |f|\ra\, dz  + e^{-M} \int_{-H}^H \la |\grad f|\ra\, dz \nonumber\\
&&\mbox{}+ H  + N^2 \int_{-H}^{H} \la |\grad  u|^2\ra\, d z + e^{-4N}  \int_{-H}^H\la|\grad u|^4\ra\, dz\label{111}.
\end{eqnarray}
Recalling the symmetry properties of the involved quantities at $z=0$, we see that we can replace the domain of vertical integration in the above estimate by $[0,H]$. Observe that the (weak) vertical derivative of $f$ at $z\in\{0,H\}$ might not be defined in the case where $\partial_{ij}=\partial_z^2$. In this case, we may interpret $\int \la |\partial_z f|\ra\, dz$ in the sense of a BV norm and use the symmetry property to estimate
\[
\int_{-H}^H \la |\del_z f|\ra\, dz \;\le\; 2\la |f_{|z=0}|\ra  + 2\la | f_{|z=H}|\ra+ 2\int_{0}^H \la |\del_z f|\ra\, dz\;\lesssim\; \int_{0}^H \la |\del_z f|\ra\, dz,
\]
where the last inequality follows from a standard Poincar\'e/trace estimate, using the fact that $\partial_z^{-1}f=0$ for $z\in\{0,H\}$.

\medskip

It remains to estimate the terms on the r.h.s.\ of \eqref{111}. We start with the first integral. By the definition of $f$ and using H\"older's inequality, we have
\begin{eqnarray*}
\lefteqn{\int_0^H \la |f|\ra\, dz}\\
&\lesssim& \frac1H\int_0^H \la|\grad^2 w|\ra\, dz + \left(\int_0^H \la|\grad u|^2\ra\, dz\right)^{1/2}\left( \int_0^H \la|\grad^2\theta|^2\ra\, dz\right)^{1/2}\\
&&\mbox{} + \left(\int_0^H \la|\grad^2 u|^{4/3}\ra\, dz\right)^{3/4}\left( \int_0^H \la|\grad\theta|^4\ra\, dz\right)^{1/4}.
\end{eqnarray*}
Via Young's and Jensen's inequalities and \eqref{108}--\eqref{109}, we further estimate
\begin{eqnarray}
\int_0^H \la |f|\ra\, dz&\lesssim& \left(\frac1H\int_0^H \la|\grad^2 w|^2\ra\, dz\right)^{1/2}\nonumber\\
&&\mbox{} +\int_0^H \la|\grad u|^2\ra\, dz +\int_0^H \la|\grad^2 u|^{4/3}\ra\, dz   + (\star)\nonumber\\
&\stackrel{\eqref{AL3.1}\&\eqref{AN4}}{\lesssim}& 1 +  H\Nu +H + (\star)\nonumber\\
&\stackrel{\eqref{110}\&H\gg1}{\lesssim}& (\ln^{\gamma} H)H + (\star)\label{112},
\end{eqnarray}
for some $\gamma>0$. Here and in the following, we denote by $(\star)$ terms that can be absorbed in the l.h.s.\ of \eqref{111}.

\medskip

For the $\grad f$-term, we first compute
\[
|\grad f|\;\lesssim\; \frac1H|\grad^3w| + |\grad u||\grad^3\theta| + |\grad^2 u||\grad^2\theta| + |\grad^3 u||\grad\theta|.
\]
We estimate using H\"older's inequality
\begin{eqnarray*}
\lefteqn{\int_0^H \la |\grad f|\ra\, dz}\\
&\lesssim& \frac1H\int_0^H \la|\grad^3 w|\ra\, dz + \left(\int_0^H \la|\grad u|^4\ra\, dz\right)^{1/4}\left( \int_0^H \la|\grad^3\theta|^{4/3}\ra\, dz\right)^{3/4}\\
&&\mbox{} + \left(\int_0^H \la|\grad^2 u|^{2}\ra\, dz\right)^{1/2}\left( \int_0^H \la|\grad^2\theta|^2\ra\, dz\right)^{1/2}\\
&&\mbox{} + \left(\int_0^H \la|\grad^3 u|^{2}\ra\, dz\right)^{1/2}\left( \int_0^H \la|\grad\theta|^2\ra\, dz\right)^{1/2}.
\end{eqnarray*}
We apply Young's and Jensen's inequalities as above and estimate with the help of \eqref{107},\eqref{109}\&\eqref{eq:7}
\begin{eqnarray*}
\int_0^H \la |\grad f|\ra\, dz
&\lesssim& \left(\frac1H\int_0^H \la|\grad^3 w|^2\ra\, dz\right)^{1/2} + \int_0^H \la|\grad u|^4\ra\, dz\\
&&\mbox{} + \int_0^H \la|\grad^2 u|^{2}\ra\, dz + \left(\int_0^H \la|\grad^3 u|^{2}\ra\, dz\right)^{1/2}\Nu^{1/2}+ (\star).
\end{eqnarray*}
To control the terms on the right, we finally invoke \eqref{AL6.1} and \eqref{AL3.1}:
\begin{eqnarray}
\int_0^H \la |\grad f|\ra\, dz
&\lesssim& \frac1{H^{1/2}}\Nu^{1/2} + H^5 + H + \Nu+ (\star)\nonumber\\
&\stackrel{\eqref{110}\&H\gg1}{\lesssim}& H^5+ (\star).\label{113}
\end{eqnarray}

\medskip

Inserting \eqref{112} and \eqref{113} into \eqref{111} and applying \eqref{AL3.1} and \eqref{AN4} to the remaining terms, we have
\begin{eqnarray*}
\int_0^H \la|\grad^4\theta|\ra\, dz &\lesssim& M (\ln^{\gamma}H)H + e^{-M} H^5 + H + N^2H\Nu + e^{-4N}H^5.
\end{eqnarray*}
Choosing $M= 4\ln H$ and $ N=\ln H$, and using \eqref{110} and $H\gg1$ again leads to the desired estimate \eqref{114}.

\end{proof}

We now turn to the

\begin{proof}[Proof of Theorem \ref{AT1}] {\em Step 1. Localization of the heat equation.} 
Let $\eta=\eta(z)$ denote a $2H$-periodic smooth function with $\eta(z)=\eta(-z)$ that has the cut-off properties $\eta=1$ for $z\in [-1,1]$, $\eta=0$ for $z\not\in[-2, 2]$, and $\sup \left(|\eta|+ |\eta'| + |\eta''| + |\eta'''|\right)\lesssim1$, where $\eta'=\frac{d}{dz}\eta$ etc. Multiplying \eqref{AT1.1} by $\eta$ yields
\begin{equation}\label{T2.1}
\partial_t(\eta\theta) + u\cdot\grad(\eta\theta) - \laplace (\eta\theta)\;=\; \frac1H\eta w + \eta' w\theta - 2\eta'\partial_z \theta- \eta''\theta.
\end{equation}

\medskip

{\em Step 2. Bound on the second-order derivatives.}
The bound on the second order derivatives relies on Proposition \ref{AP3}. We define $\zeta=\eta\theta$ and $f=\frac1H\eta w + \eta' w\theta - 2\eta'\partial_z\theta - \eta''\theta$ and choose $\delta=2$. Eventually after mollifying $\zeta$ and $f$ and recalling the symmetry of the functions involved, \eqref{AP3.2} yields
\begin{equation}\label{T2.2}
\int_0^H \la|\grad^2 (\eta\theta)|^2\ra\, dz\;\lesssim\; \int_0^H \la f^2\ra\, dz + \sup_x|\eta\theta|\int_0^2 \la|\grad u|^2\ra\, dz.
\end{equation}
Also notice that the interpolation inequality \eqref{L10.2} together with \eqref{AT1.4} yields the lower bound
\begin{equation}\label{T2.4}
\int_0^H \la |\grad(\eta\theta)|^{4}\ra\, dz\;\lesssim\; \int_0^H \la|\grad^2 (\eta\theta)|^2\ra\, dz.
\end{equation}
To bound the r.h.s.\ of \eqref{T2.2}, we invoke the definition of $\eta$ and compute
\begin{eqnarray*}
f^2&\lesssim& \frac1{H^2}w^2 + w^2\theta^2 + (\partial_z\theta)^2 + \theta^2\\
&\stackrel{\eqref{AT1.4}\& H\gg1}{\lesssim}& w^2 + (\partial_z\theta)^2 + 1,
\end{eqnarray*}
and the r.h.s.\ is in vertical direction supported in $[0,2]$. Hence, using \eqref{AT1.4} and the properties of $\eta$ again, \eqref{T2.2} becomes
\[
\int_0^1 \la |\grad^2\theta|^2\ra\, dz\;\lesssim\; \int_0^2 \la w^2\ra\, dz + \int_0^2 \la(\partial_z\theta)^2\ra\, dz + \int_0^2\la|\grad u|^2\ra\, dz + 1,
\]
and thus, by \eqref{Aeq:3}, \eqref{eq:7} and \eqref{110}, we obtain
\begin{equation}\label{T2.3}
\int_0^1 \la|\grad^2\theta|^2\ra\,dz\;\lesssim\; \ln^{\gamma} H,
\end{equation}
where $\gamma>0$ is an arbitrary constant dependent only on the space dimension, and whose value may change from line to line in the remainder of the proof. In view of \eqref{T2.4}, the same estimate yields
\begin{equation}\label{T2.5}
\int_0^1 \la|\grad\theta|^4\ra\,dz\;\lesssim\; \ln^{\gamma} H.
\end{equation}
For the following steps, it is crucial to remark that the estimates in \eqref{T2.3} and \eqref{T2.5} hold in {\em any} strip $[0,\Lambda)^{d-1}\times(0,h)$ for $h\sim 1$.

\medskip

{\em Step 3. Bound on the third-order derivatives.}
The bound on the third-order derivatives relies on Proposition \ref{AP2} with $r=4/3$. Starting point is \eqref{T2.1}. Differentiating this localized equation w.r.t.\ $x_i$ for some $i\in\{1,\dots, d\}$ yields
\begin{eqnarray}
\lefteqn{\partial_t \partial_i (\eta\theta) + u\cdot\grad \partial_i(\eta\theta) - \laplace\partial_i (\eta\theta)}\nonumber\\
 &=& \frac1H \partial_i \eta w + \frac1H \eta \partial_i w  + \partial_i\eta' w\theta  + \eta' w \partial_i\theta\nonumber\\
&&\mbox{} -2\partial_i\eta' \partial_z \theta - 2\eta'\partial_i\partial_z\theta - \partial_i \eta''\theta - \eta''\partial_i\theta  - \eta \partial_i u\cdot\grad \theta  . \label{T2.11}
\end{eqnarray}
With $\zeta=\partial_i(\eta\theta)$ and $f$ defined by the r.h.s.\ of this equation, and $\delta=2$, we may apply Proposition \ref{AP2}. Eventually using a mollification argument and using the symmetry of $\zeta$, $u$, and $f$, \eqref{AP2.0dis} yields
\begin{eqnarray*}
\lefteqn{\left(\int_0^H \la|\grad^2\partial_i (\eta\theta)|^{4/3}\ra\, dz\right)^{3/4}}\\
&\lesssim& M\left(\int_0^H \la|f|^{4/3}\ra\, dz\right)^{3/4}\\
&&\mbox{} + e^{-M}\left(\int_0^H \la|\grad f|^{4/3}\ra\, dz\right)^{3/4}\\
&&\mbox{} + \left(\int_0^H \la(\partial_i(\eta\theta))^2\ra\, dz\right)^{1/2}\\
&&\mbox{} + N\left(\int_0^4 \la|\grad u|^{2}\ra\, dz\right)^{1/2}\left(\int_0^H \la(\partial_i(\eta\theta))^4\ra\, dz\right)^{1/4}\\
&&\mbox{} + e^{-N}\left(\int_0^H \la|\grad u|^{4}\ra\, dz\right)^{1/4}\left(\int_0^H \la|\grad\partial_i(\eta\theta)|^2\ra\, dz\right)^{1/2}.
\end{eqnarray*}
Summing over all $i$, invoking the interpolation inequalities \eqref{L10.2}\&\eqref{L10.1} together with \eqref{AT1.4} and using Young's inequality, the above estimate simplifies to
\begin{eqnarray}
\lefteqn{\left(\int_0^1 \la|\grad^3 \theta|^{4/3}\ra\, dz\right)^{3/4}}\nonumber\\
&\lesssim& M\left(\int_0^H \la|f|^{4/3}\ra\, dz\right)^{3/4} + e^{-M}\left(\int_0^H \la|\grad f|^{4/3}\ra\, dz\right)^{3/4}\label{T2.6}\\
&&\mbox{} + N^{3/2}\left(\int_0^4 \la|\grad u|^{2}\ra\, dz\right)^{3/4} + e^{-3N}\left(\int_0^H \la|\grad u|^{4}\ra\, dz\right)^{3/4}+1.\label{T2.7}
\end{eqnarray}
The velocity terms in \eqref{T2.7} are easily controlled. Indeed, we apply \eqref{Aeq:3} and \eqref{AL3.1} and get
\begin{eqnarray}
\lefteqn{N^{3/2} \left(\int_0^4 \la|\grad u|^{2}\ra\, dz\right)^{3/4} + e^{-3N}\left(\int_0^H \la|\grad u|^{4}\ra\, dz\right)^{3/4}+1}\nonumber\\
&\lesssim& N^{3/2} \ln^{\gamma} H + e^{-3N} H^{15/4} + 1\hspace{11em}\nonumber\\
&\lesssim& \ln^{\gamma} H\label{T2.8},
\end{eqnarray}
if we choose $N=\frac54\ln H$ and use $H\gg1$ in the last inequality. We now turn to the estimate of the terms in \eqref{T2.6}. By the definition of $f$ and $\eta$ and using the maximum principle for the temperature, i.e., \eqref{AT1.4}, and $H\gg1$, we have
\[
|f|\;\lesssim\; |u| + |\grad u| + |u||\grad\theta| + |\grad u||\grad\theta| + |\grad\theta| + |\grad^2\theta|+1,
\]
and $f$ is supported in $[0,2]$. Therefore, applying Jensen's and Young's inequality, we estimate
\begin{eqnarray*}
\lefteqn{\left(\int_0^H \la|f|^{4/3}\ra\, dz\right)^{3/4}}\\
&\lesssim& \left(\int_0^2 \la |u|^{2}\ra\, dz\right)^{1/2} + \left(\int_0^2 \la |\grad u|^{2}\ra\, dz\right)^{1/2}\\
&&\mbox{}+\left(\int_0^2 \la |\grad\theta|^{4}\ra\, dz\right)^{1/4} + \left(\int_0^2 \la |\grad^2\theta|^{2}\ra\, dz\right)^{1/2}+1.
\end{eqnarray*}
It remains to invoke \eqref{Aeq:3}, \eqref{T2.3},\eqref{T2.5}, and $H\gg1$ to deduce
\begin{equation}\label{T2.9}
\left(\int_0^H \la|f|^{4/3}\ra\, dz\right)^{3/4} \;\lesssim\; \ln^{\gamma}H.
\end{equation}
We finally investigate the second term in \eqref{T2.6}. Proceeding as in the previous computation, but slightly more crudely, we estimate
\begin{eqnarray*}
|\grad f| &\lesssim& |u| + |u||\grad\theta|+ |u||\grad^2 \theta| + |\grad u| + |\grad u ||\grad\theta| + |\grad u||\grad^2\theta|\\
&&\mbox{} + |\grad^2 u|+ |\grad^2 u||\grad\theta| + |\grad\theta|+  |\grad^2\theta| + |\grad^3\theta|+1. 
\end{eqnarray*}
We skip the details of how to estimate the $L_{4/3}$ norm of $\grad f$ step-by-step. Via standard inequalities and the global bounds \eqref{AL3.1} and \eqref{AL9.2}, we can easily derive a bound of the form
\begin{equation}\label{T2.10}
\left(\int_0^H \la|\grad f|^{4/3}\ra\, dz\right)^{3/4} \;\lesssim\; H^q
\end{equation}
for some $q>0$. Hence, choosing $M=q\ln H$ and substituting \eqref{T2.8}--\eqref{T2.10} into \eqref{T2.6}\&\eqref{T2.7} yields
\[
\left(\int_0^1 \la|\grad^3 \theta|^{4/3}\ra\, dz\right)^{3/4}\; \lesssim\; \ln^{\gamma} H .
\]

\medskip

{\em Step 4. Bounds on the fourth-order derivatives.} 
The derivation of the bound on the fourth-order derivatives proceeds very much along the lines of that on the third-order derivatives. We differentiate \eqref{T2.11} w.r.t.\ $x_j$ for some $j\in\{1,\dots,d\}$:
\[
\partial_t\partial_{ij}(\eta\theta) + u\cdot\grad\partial_{ij}(\eta\theta)-\laplace\partial_{ij}(\eta\theta)\;=\; f,
\]
where $f$ is such that
\[
|f| \;\lesssim\; |u|^2 + |\grad u|^2 + |\grad^2 u|^2  + |\grad\theta|^4 + |\grad^2\theta|^2  + |\grad^3\theta|^{4/3} + 1,
\]
and
\[
|\grad f|\; \lesssim\; |u|^4 + |\grad u|^4 + |\grad^2 u|^2+ |\grad^3 u|^2+ |\grad\theta|^2 + |\grad^2\theta|^2 + |\grad^3\theta|^{4/3} + |\grad^4\theta|+ 1,
\]
and $f$ is supported in $[0,2]$. Eventually after mollification, we may apply Proposition \ref{AP2} with $r=1$ and $\delta=2$ and thus \eqref{AP2.0dis} yields, recollecting the symmetries in $z=0$:
\begin{eqnarray*}
\lefteqn{\int_0^H \la|\grad^2 \partial_{ij}(\eta\theta)|\ra\, dz}\\
&\lesssim&M\int_0^H \la|f|\ra\, dz\\
&&\mbox{} +e^{-M} \int_0^H \la|\grad f|\ra\, dz\\
&&\mbox{}+\left(\int_0^H \la( \partial_{ij}(\eta\theta))^2\ra\, dz\right)^{1/2}\\
&&\mbox{}+N\left(\int_0^4 \la|\grad u|^2\ra\, dz\right)^{1/2}\left(\int_0^H \la( \partial_{ij}(\eta\theta))^2\ra\, dz\right)^{1/2}\\
&&\mbox{}+e^{-N}\left(\int_0^4 \la|\grad u|^4\ra\, dz\right)^{1/4}\left(\int_0^H \la|\grad \partial_{ij}(\eta\theta)|^{4/3}\ra\, dz\right)^{3/4}.
\end{eqnarray*}
In the case where $\partial_z f$ is not defined as a weak derivative, we argue as in the proof of Lemma \ref{AL9}. Summing over all $i,j$ and estimating with the help of \eqref{L10.1},\eqref{AT1.4},\eqref{L9.1} and Young's inequality, we obtain
\begin{eqnarray*}
\lefteqn{\int_0^H \la|\grad^2 \partial_{ij}(\eta\theta)|\ra\, dz}\\
&\lesssim&M\int_0^H \la|f|\ra\, dz +e^{-M} \int_0^H \la|\grad f|\ra\, dz\\
&&\mbox{}+N^2\int_0^4 \la|\grad u|^2\ra\, dz +e^{-4N}\int_0^4 \la|\grad u|^4\ra\, dz + 1.
\end{eqnarray*}
Local bounds on the $\grad^2 u$-terms can be obtained via the Ehrling estimate
\[
\int_0^4 \la|\grad^2 u|^2\ra\, dz\;\lesssim\; \int_0^4 \la|\grad u|^2\ra\, dz +\int_0^4 \la|\grad^3 u|^2\ra\, dz.
\]
Proceeding as in the previous steps and observing that the third-order derivatives of $u$ can be controlled via \eqref{AL6.1} and \eqref{110}, we conclude that
\[
\int_0^1\la|\grad^4\theta|\ra\, dz \;\lesssim\; M\ln^{\gamma} H + e^{-M}H^q + N^2\ln^{\gamma} H + e^{-4N}H^5 +1.
\]
We choose $M=q\ln H$, $N=\frac54\ln H$ and deduce
\[
\int_0^1 \la|\grad^4 \theta|\ra\, dz\;\lesssim\;\ln^{\gamma}H.
\]
\end{proof}

We finally turn to the proofs of Propositions \ref{AP2} and \ref{AP3}. The proof of the latter one in standard. We display the easy argument for the convenience of the reader:

\begin{proof}[Proof of Proposition \ref{AP3}.]
We differentiate \eqref{AP3.1} w.r.t. the $x_i$ coordinate, e.g.,
\[
\partial_t\partial_i \zeta + u\cdot \grad \partial_i \zeta - \laplace \partial_i \zeta\;=\; \partial_i f - \partial_i u\cdot \grad \zeta.
\]
We multiply this equation by $\partial_i\zeta$ and integrate in space and time:
\[
- \int_{-H}^H \la \partial_i\zeta \laplace\partial_i \zeta\ra\, dz \;\le\; \int_{-H}^H \la \partial_i\zeta \partial_i f\ra\, dz -\int_{-H}^H \la \partial_i \zeta\partial_i u\cdot \grad  \zeta \ra\, dz.
\]
Here we have used that $\la \partial_t (\partial_i\zeta)^2\ra\ge0$ due to the long-time averages. Moreover, the transport term drops out because of  \eqref{AP3.1bis}. Integration by parts and \eqref{AP3.1bis} yield
\[
 \int_{-H}^H \la |\grad\partial_i\zeta |^2 \ra\, dz \;\le\; - \int_{-H}^H \la \partial_i^2 \zeta  f\ra\, dz  + \int_{-H}^H \la\grad \partial_i \zeta\cdot \partial_i u   \zeta \ra\, dz.
\]
It remains to apply the H\"older and Young inequalities and sum over $1\le i\le d$ to obtain \eqref{AP3.2} via \eqref{AP3.1cis}.
\end{proof}

The proof of Proposition \ref{AP2} requires some preparation. We first provide new notation. For a Schwartz function $\psi(x)$ in $\R^d$, we define its Fourier transform $\F \psi(q)$, $q\in\R^d$, via
\[
\F\psi(q):= \int_{\R^d}\psi(x) e^{iq\cdot x}\, dx.
\]
For a periodic function $\zeta(x)$ with cell of periodicity $Q=[0,a_1]\times\dots\times[0,a_d]$ its Fourier transform $\F\zeta(q)$, $q\in 2\pi/a_1\Z\times\dots\times2\pi/a_d\Z$, is defined by
\[
\F \zeta(q)\;:=\; \frac1{|Q|} \int_Q \zeta(x)e^{iq\cdot x}\, dx.
\]

\medskip

Next, we derive an energy-estimate for narrow-banded solutions to convection-diffusion equations.

\begin{lemma}\label{AL7}
Let $\zeta$, $u$, and $f$ be periodic functions with vertical period $2H$ satisfying
\begin{equation}\label{AL7.0}
\del_t \zeta + u \cdot \grad\zeta - \laplace \zeta\; =\; f.
\end{equation}
Assume that $\zeta$ is narrow-banded in Fourier space in the sense that
\begin{equation}\label{AL7.0bis}
\F \zeta (q)\; =\; 0 \mbox{ for all }q\not\in B_{\sigma}(q_0)
\end{equation}
for some $e^{-1} <|q_0|\le e$ and $0<\sigma\ll1$. For $1\le r<\infty$ it holds
\begin{equation}\label{AL7.1}
\int_{-H}^H \la |\zeta|^{r}\ra\, dz
\;\lesssim\;
\int_{-H}^H \la |f|^{r}\ra\, dz + \int_{-H}^H \la|\div u||\zeta|^{r}\ra\, dz.
\end{equation}
\end{lemma}

\begin{proof}[Proof of Lemma \ref{AL7}] Let $A(s)$ be a smooth approximation of
\[
A(s)\; =\; \frac1r|s|^{r}.
\]
In view of \eqref{AL7.0}, we have that $\partial_t A(\zeta) + u\cdot \grad A(\zeta) - \laplace \zeta A'(\zeta)= f A'(\zeta)$, and thus after integration, integration by parts, and time averaging:
\[
- \int_{-H}^H \la \laplace \zeta A'(\zeta)\ra\, dz\; \le\; \int_{-H}^H \la (\div u)A(\zeta)\ra\, dz + \int_{-H}^H \la fA'(\zeta)\ra\, dz.
\]
Notice that the term involving the time derivative drops out in the time average, since $\la \del_t A(\zeta)\ra\ge0$ . Now we may carry out the approximation for $A$. Obviously, the assertion in \eqref{AL7.1} follows from
\begin{equation}\label{AL7.2}
-\int_{-H}^H \la \laplace \zeta A'(\zeta)\ra\, dz 
\;=\;
 - \int_{-H}^H \la \laplace \zeta (\sign \zeta)|\zeta|^{r-1}\ra\, dz
\;\gtrsim\;
\int_{-H}^H \la|\zeta|^{r}\ra\, dz
\end{equation}
and
\begin{eqnarray*}
\int_{-H}^H \la fA'(\zeta)\ra\, dz
&=&
\int_{-H}^H \la f(\sign \zeta)|\zeta|^{r-1}\ra\, dz\\
& \le& \left(\int_{-H}^H \la |f|^{r}\ra\, dz\right)^{1/r}\left( \int_{-H}^H\la |\zeta|^{r}\ra\, dz\right)^{(r-1)/r}
\end{eqnarray*}
via Young's inequality (if $r>1$).

\medskip

In our argumentation for \eqref{AL7.2}, we follow \cite[Lemma 1]{OttoRamos} (or \cite[Proposition 2]{Otto09}). We first show that
\begin{equation}\label{AL7.3}
\int_{-H}^H \la | -\laplace \zeta - |q_0|^2 \zeta|^{r} \ra\, dz 
\; \lesssim\; \sigma^{r} \int_{-H}^H \la |\zeta|^{r}\ra\, dz.
\end{equation}
For this purpose, we select a Schwartz function $\psi$ that satisfies $(\F \psi)(q) \;=\; 1$ for $|q|\le1$ and define
\[
\psi_{\sigma}(x)\; =\; \sigma^d \psi(\sigma x)e^{-i q_0\cdot x}.
\]
An easy calculation shows that $(\F \psi_{\sigma})(q)=1$ for $q\in B_{\sigma}(q_0)$. Therefore, by the narrow-bandedness assumption \eqref{AL7.0bis}, $\psi_{\sigma}$ leaves $\zeta$ invariant under convolution, i.e.,
\[
\zeta = \psi_{\sigma}\ast\zeta.
\]
It follows that
\[
-\laplace \zeta - |q_0|^2 \zeta\; =\; (-\laplace \psi_{\sigma} - |q_0|^2\psi_{\sigma})\ast \zeta,
\]
and thus, applying the convolution estimate yields
\[
\int_{-H}^H \la |-\laplace \zeta - |q_0|^2 \zeta|^{r}\ra\, dz
\;\le\;
\left(\int_{\R^d} |-\laplace \psi_{\sigma} - |q_0|^2\psi_{\sigma}|\, dx\right)^{r} \int_{-H}^H \la |\zeta|^{r}\ra\, dz.
\]
For \eqref{AL7.3}, we have to show that
\[
\int_{\R^d} |-\laplace \psi_{\sigma} - |q_0|^2\psi_{\sigma}|\, dx
\;\lesssim\; \sigma.
\]
Indeed, since $-\laplace (e^{-iq_0\cdot x})= |q_0|^2 e^{-iq_0\cdot x}$, we obtain
\[
(-\laplace \psi_{\sigma} - |q_0|^2\psi_{\sigma})(x)
\;=\; \left(2i\sigma^{d+1} q_0 \cdot (\grad \psi)(\sigma x) - \sigma^{d+2} (\laplace\psi)(\sigma x)\right) e^{-iq_0\cdot x },
\]
so that, because of $|q_0|\le e$ and $\sigma < 1$ and since $\psi$ is a Schwartz function, we have
\[
\int_{\R^d} |-\laplace \psi_{\sigma} - |q_0|^2\psi_{\sigma}|\, dx
\;\lesssim\;
\sigma |q_0| \int_{\R^d} |\grad \psi|\, dx + \sigma^2 \int_{\R^d} |\grad^2\psi|\, dx
\;\lesssim\;
\sigma.
\]
It remains to argue how \eqref{AL7.3} implies \eqref{AL7.2}. It is
\begin{eqnarray*}
\lefteqn{- \int_{-H}^H \la (\sign \zeta) |\zeta|^{r-1} \laplace \zeta\ra\, dz}\\
&=& \int_{-H}^H \la (\sign \zeta) |\zeta|^{r-1} |q_0|^2 \zeta\ra\, dz + \int_{-H}^H \la (\sign \zeta) |\zeta|^{r-1}(-\laplace \zeta -  |q_0|^2 \zeta)\ra\, dz \\
&\ge& |q_0|^2 \int_{-H}^H \la|\zeta|^{r}\ra\, dz\\
&&\mbox{}- \left(\int_{-H}^H \la|\zeta|^{r}\ra\, dz\right)^{(r-1)/r}\left(\int_{-H}^H \la | -\laplace \zeta - |q_0|^2 \zeta|^{r} \ra\, dz \right)^{1/r}\\
&\stackrel{\eqref{AL7.3}}{\ge}& (|q_0|^2 - C\sigma)\int_{-H}^H \la |\zeta|^{r}\ra\, dz
\end{eqnarray*}
for some constant $C>0$. We choose $\sigma$ small enough, such that $2C\sigma\le e^{-2}$ and obtain \eqref{AL7.2} because of $|q_0|\ge e^{-1}$.
\end{proof}

Another tool in the proof of Proposition \ref{AP2} are the following commutator estimates:

\begin{lemma}\label{AL8}
Let $[u\cdot, \phi\ast]\xi$ denote the commutator of the operations ``multiplication with $u$'' and ``convolution with $\phi$'', that is,
\[
[u\cdot, \phi\ast]\xi\; =\; u\cdot(\xi \ast \phi) - (u\cdot \xi)\ast \phi.
\]
Then we have the estimates
\begin{eqnarray}
\lefteqn{\left(\int_{-H}^H \la |[u\cdot, \phi\ast]\grad\zeta|^{r}\ra\, dz\right)^{1/r}}\nonumber\\
&\lesssim&\left(\int_{\R^d} |\grad \phi||\tilde x|\, d\tilde x +\int_{\R^d} |\phi|\, d\tilde x \right)\nonumber\\
&&\mbox{}\times
 \left(\int_{-H}^H \la |\zeta|^{rp/(p-r)}\ra\, dz\right)^{(p-r)/(rp)}\left(\int_{-H}^H \la |\grad u|^p\ra\, dz\right)^{1/p},\label{AL8.1}
\end{eqnarray}
and
\begin{eqnarray}
\lefteqn{\left(\int_{-H}^H \la |[u\cdot, \phi\ast]\grad\zeta|^{r}\ra\, dz\right)^{1/r}}\nonumber\\
&\lesssim&\left(\int_{\R^d} | \phi||\tilde x|\, d\tilde x\right)\nonumber\\
&&\mbox{}\times
 \left(\int_{-H}^H \la |\grad \zeta|^{rp/(p-r)}\ra\, dz\right)^{(p-r)/(rp)}\left(\int_{-H}^H \la |\grad u|^p\ra\, dz\right)^{1/p},\label{AL8.2}
\end{eqnarray}
for any $1\le r<\infty$ and $1<p<\infty$.
\end{lemma}

\begin{proof}[Proof of Lemma \ref{AL8}] We only prove the first estimate. The second one is shown similarly. We start with a pointwise statement:
\begin{eqnarray*} 
\lefteqn{ \left(u\cdot(\grad\zeta \ast \phi) - (u\cdot \grad\zeta)\ast \phi\right)(x)}\\
&=& \int_{\R^d}\phi(\tilde x) \grad_x\zeta(x-\tilde x)\cdot(u(x) - u(x-\tilde x))\, d\tilde x\\
&=&  \int_{\R^d}\grad_{\tilde x}\phi(\tilde x)\cdot \zeta(x-\tilde x)(u(x) - u(x-\tilde x))\, d\tilde x\\
&&\mbox{} + \int_{\R^d}\phi(\tilde x) \zeta(x-\tilde x)(\grad_x\cdot u)(x-\tilde x)\, d\tilde x\\
&\le&  \int_{\R^d}|\grad\phi(\tilde x)|| \zeta(x-\tilde x)||\tilde x|\int_0^1|\grad u(x-s\tilde x)|\, dsd\tilde x\\
&&\mbox{} + \int_{\R^d}|\phi(\tilde x)|| \zeta(x-\tilde x)||(\grad\cdot u)(x-\tilde x)|\, d\tilde x.
\end{eqnarray*}
This implies by the triangle inequality
\begin{eqnarray*}
\lefteqn{\left(\int_{-H}^H \la |[u\cdot, \phi\ast]\grad\zeta|^{r}\ra\, dz\right)^{1/r}}\\
&\le& \left(\int_{-H}^H \left\la\left(\int_{\R^d}|\grad\phi(\tilde x)|| \zeta(\tacka-\tilde x)||\tilde x|\int_0^1|\grad u(\tacka-s\tilde x)|\, dsd\tilde x\right)^r\right\ra\, dz\right)^{1/r}\\
&&\mbox{} +\left(\int_{-H}^H \left\la\left(\int_{\R^d}|\phi(\tilde x)|| \zeta(\tacka-\tilde x)||(\grad\cdot u)(\tacka-\tilde x)|\, d\tilde x\right)^r\right\ra\, dz\right)^{1/r}.
\end{eqnarray*}
We consider the first integral only. The estimate of the second one is obtained in an analogous way. We have via H\"older inequality and Fubini's Theorem
\begin{eqnarray*}
\lefteqn{ \int_{-H}^H \left\la \left(\int_{\R^d}|\grad\phi(\tilde x)|| \zeta(\tacka-\tilde x)||\tilde x|\int_0^1|\grad u(\tacka-s\tilde x)|\, dsd\tilde x\right)^r\right\ra \, dz}\\
&\le& \left(\int_{\R^d}|\grad\phi||\tilde x|\, d\tilde x\right)^{r-1}\\
&&\mbox{}\times
\int_{-H}^H \left\la  \int_{\R^d}|\grad\phi(\tilde x)||\tilde x|| \zeta(\tacka-\tilde x)|^{r}\int_0^1|\grad u(\tacka-s\tilde x)|^{r}\, dsd\tilde x\right\ra\, dz\\
&=&\left(\int_{\R^d}|\grad\phi||\tilde x|\, d\tilde x\right)^{r-1}\\
&&\mbox{}\times
\int_{\R^d}|\grad\phi(\tilde x)||\tilde x|\int_0^1\int_{-H}^H \left\la  | \zeta(\tacka-\tilde x)|^{r}|\grad u(\tacka-s\tilde x)|^{r}\, \right\ra\, dzdsd\tilde x.
\end{eqnarray*}
We use the H\"older inequality again and the translation invariance to estimate the inner integrals:
\begin{eqnarray*}
\lefteqn{\int_{-H}^H\left\la  | \zeta(\tacka-\tilde x)|^{r}|\grad u(\tacka-s\tilde x)|^{r}\, \right\ra\, dz}\\
&\le& \left(\int_{-H}^H \la |\zeta(\tacka-\tilde x)|^{rp/(p-r)}\ra\, dz\right)^{(p-r)/p} \left(\int_{-H}^H\la |\grad u(\tacka-s\tilde x)|^p \ra\, dz\right)^{r/p}\\
&=&\left(\int_{-H}^H \la |\zeta|^{rp/(p-r)}\ra\, dz\right)^{(p-r)/p}\left(\int_{-H}^H \la |\grad u|^p\ra\, dz\right)^{r/p}.
\end{eqnarray*}
We combine the above estimates and get
\begin{eqnarray*}
\lefteqn{ \int_{-H}^H \left\la \left(\int_{\R^d}|\grad\phi(\tilde x)|| \zeta(\tacka-\tilde x)||\tilde x|\int_0^1|\grad u(\tacka-s\tilde x)|\, dsd\tilde x\right)^r\right\ra\, dz}\\
&\le& \left(\int_{\R^d}|\grad\phi||\tilde x|\, d\tilde x\right)^{r}
 \left(\int_{-H}^H \la |\grad u|^p\ra\, dz\right)^{r/p}\left(\int_{-H}^H \la |\zeta|^{rp/(p-r)}\ra\, dz\right)^{(p-r)/p}.
\end{eqnarray*}
\end{proof}

We are now in the position to prove Proposition \ref{AP2}.
\begin{proof}[Proof of Proposition \ref{AP2}]
The proof relies on Besov spaces rather than Sobolev spaces. We introduce a (non-dyadic) Littlewood--Paley decomposition. Let $\{\phi_{\ell}(x)\}_{\ell\in \Z}$ be a family of Schwartz functions with the following properties:
\begin{eqnarray}
(\F \phi_0)(q) &\not=& 0\mbox{ only for $q$ with } |q|\in (e^{-1},e), \label{ALP1}\\
(\F \phi_{\ell})(q) &=& (\F \phi_0)(e^{-\ell}q)\mbox{ for all $\ell$ and $q$}, \label{ALP2}\\
\sum_{\ell \in \Z} (\F \phi_{\ell})(q) &=&1\mbox{ for all }q\not= 0.\label{ALP3}
\end{eqnarray}
We refer to \cite[6.1.7\ Lemma]{Bergh} for a construction with dyadic blocks, which easily adapts to the non-dyadic case. The Littlewood--Paley decomposition  $\{\varphi_{\ell}\}_{\ell\in \Z}$ of a periodic function $\varphi$ is defined via
\[
\varphi_{\ell}\; =\; \phi_{\ell} \ast \varphi.
\]
Notice that \eqref{ALP1}--\eqref{ALP3} imply that $\varphi = \la\varphi\ra +\sum_{\ell\in\Z}\varphi_{\ell}$. We also introduce the low-pass filter
\[
\varphi_{0}^{<}\; =\; \sum_{\ell\le -1} \varphi_{\ell}.
\]
We now turn to the proof of \eqref{AP2.0dis}. We invoke the triangle inequality to split and estimate the integral on the l.h.s.\ according to
\begin{eqnarray}
\left(\int_{-H}^H \la |\grad^2 \zeta|^r\ra\, dz\right)^{1/r} &\stackrel{\eqref{AP2.1cis}}{=}& \left(\int_{-\delta}^{\delta} \la |\grad^2 \zeta|^r\ra\, dz\right)^{1/r}\nonumber\\
&\le&\left(\int_{-\delta}^{\delta} \la |\grad^2 \zeta_0^{<}|^r\ra\, dz\right)^{1/r} \label{AP2.2}\\
&&\mbox{} +\sum_{\ell=0}^{N-1}\left(\int_{-H}^{H} \la |\grad^2 \zeta_{\ell}|^r\ra\, dz\right)^{1/r} \label{AP2.3}\\
&&\mbox{}+\sum_{\ell\ge N}\left(\int_{-H}^{H} \la |\grad^2 \zeta_{\ell}|^r\ra\, dz\right)^{1/r} .\label{AP2.4}
\end{eqnarray}

\medskip

We start with the low frequency part, i.e.,  \eqref{AP2.2}. It follows from \eqref{ALP1}--\eqref{ALP3} that
\[
\zeta_0^<\; =\; \phi_0^<\ast\zeta,
\]
where $\phi_0^<$ denotes a Schwartz function with $(\F \phi_0^<)(q)=0$ for $|q|\ge e^{-1}$. This implies in particular
\[
\int |\phi_0^<|\, dx\; <\; \infty
\]
and
\[
|\F( \grad^2 (\phi_0^<\ast \zeta))(q)|\;=\; |q|^2|\F(\phi_0^<\ast \zeta)(q)| \;\le\; |\F ( \phi_0^<\ast \zeta)(q)|,
\]
so that by Plancherel's Theorem and the convolution estimate
\begin{eqnarray*}
\int_{-H}^H \la |\grad^2 \zeta_0^<|^2\ra\, dz &=& \int_{-H}^H \la |\grad^2(\phi_0^<\ast\zeta)|^2\ra\, dz\\
& \le& \int_{-H}^H \la (\phi_0^<\ast\zeta)^2\ra\, dz \\
&\lesssim& \int_{-H}^H \la \zeta^2\ra\, dz.
\end{eqnarray*}
We conclude that
\begin{eqnarray}
\int_{-\delta}^{\delta} \la |\grad^2 \zeta_0^{<}|^r\ra\, dz
&\le& 
\delta^{(2-r)/2}\left(\int_{-\delta}^{\delta} \la |\grad^2 \zeta_0^{<}|^2\ra\, dz \right)^{r/2}\nonumber\\
&\lesssim& \delta^{(2-r)/2}\left(\int_{-H}^H \la \zeta^2\ra\, dz\right)^{r/2}.\label{AP2.5}
\end{eqnarray}

\medskip

We now address the intermediate frequencies, i.e., \eqref{AP2.3}. Our goal is the following estimate:
\begin{eqnarray}
\lefteqn{\sum_{\ell=0}^{N-1}\left(\int_{-H}^H \la |\grad^2\zeta_{\ell}|^r\ra\, d z\right)^{1/r}}\nonumber\\
&\lesssim&
\sum_{\ell=0}^{N-1} \left(\int_{-H}^H \la | f_{\ell}|^r\ra\, d z\right)^{1/r}\nonumber\\
&&\mbox{} + N \left(\int_{-2\delta}^{2\delta} \la |\grad  u|^2\ra\, d z \right)^{1/2}\left(\int_{-H}^H \la  |\zeta|^{2r/(2-r)}\ra\, d z\right)^{(2-r)/2r}.\label{AP2.11}
\end{eqnarray}
Because of \eqref{AP2.1cis}, equation \eqref{AP2.1} is localized in the strip $[-\delta,\delta]\times\R^{d-1}$, so that we may replace $u$ in \eqref{AP2.1} by $\tilde u=  \eta u$, where $\eta=\eta(z)$ denotes a $2H$-periodic extension of a smooth cut-off function satisfying $\eta=1$ for $z\in[-\delta,\delta]$, $\eta=0$ for $z\not\in[-2\delta,2\delta]$, $\sup |\eta|\lesssim1$, and $\sup|\frac{d}{dz}\eta|\lesssim\frac1{\delta}$. We remark that the convection-diffusion equation \eqref{AP2.1} is invariant under the scaling
\[
x=e^{-\ell} \hat x,\quad
t=e^{-2\ell} \hat t,\quad
\zeta = e^{-2\ell}\hat \zeta,\quad
\tilde u = e^{\ell}\hat{ \tilde u},\quad
f = \hat f,
\]
i.e., it holds
\begin{equation}\label{AP2.8}
\partial_{\hat t}\hat \zeta + \hat{\tilde u} \cdot \hat\grad\hat \zeta - \hat\laplace\hat \zeta
\;=\;
\hat f
\end{equation}
in the rescaled domain of height $\hat H = e^{\ell} H$. In a first step, we apply Lemma \ref{AL7} to this rescaled equation. Given some $0<\sigma\ll1$ as in the hypothesis of Lemma \ref{AL7}, we select a finite number of open balls $B_{\sigma}(q_j)_{1\le j\le J}$ that cover the annulus $\{e^{-1} < |q|\le e\}$ and select a family of Schwartz functions $\{\psi_{\sigma, q_j}\}_{1\le j\le J}$ that form a partition of unity subordinate to that covering. We apply $ \psi_{\sigma, q_j}\ast\phi_0\ast$ to \eqref{AP2.8} and write the resulting equation as
\begin{equation}\label{AP2.12}
\partial_{\hat t}\hat \zeta_{\sigma, q_j} + \hat{\tilde u}\cdot\hat\grad\hat\zeta_{\sigma, q_j} -\hat \laplace\hat \zeta_{\sigma, q_j}\; =\;\hat f_{\sigma, q_j} + [\hat{\tilde u}\cdot, \psi_{\sigma, q_j}\ast\phi_0\ast]\hat\grad \hat\zeta,
\end{equation}
where subscript ``$\sigma, q_j$'' indicates the convolution with $\psi_{\sigma, q_j}\ast\phi_0$ and $[u\cdot, \phi\ast]\xi$ denotes the commutator of the operations ``multiplication with $u$'' and ``convolution with $\phi$'', that is,
\[
[u\cdot, \phi\ast]\xi\; =\; u\cdot(\xi \ast \phi) - (u\cdot \xi)\ast \phi.
\]
Since $ \zeta_{\sigma, q_j}$ is narrow-banded in Fourier space in the sense of \eqref{AL7.0bis},
\[
(\F\zeta_{\sigma, q_j})(q)\; =\; (\F \psi_{\sigma, q_j})(q) (\F \phi_0)(q)(\F\zeta)(q)=0\quad\mbox{for all }q\not\in B_{\sigma}(q_j),
\]
we infer from \eqref{AL7.1} that
\begin{eqnarray*}
\lefteqn{\int_{-\hat H}^{\hat H} \la |\hat\zeta_{\sigma, q_j}|^r\ra\, d\hat z}\\
&\lesssim&
\int_{-\hat H}^{\hat H} \la |\hat f_{\sigma, q_j}|^r\ra\, d\hat z\\
&&\mbox{}+ \int_{-\hat H}^{\hat H} \la |[\hat{\tilde u}\cdot, \psi_{\sigma, q_j}\ast\phi_0\ast]\hat \grad \hat\zeta|^r\ra\, d\hat z + \int_{-\hat H}^{\hat H} \la |\hat\grad\cdot\hat{\tilde u}||\hat\zeta_{\sigma, q_j}|^r\ra\, d\hat z.
\end{eqnarray*}
Applying the H\"older inequality and $|\hat\grad\cdot \hat{\tilde u}|\lesssim |\hat\grad \hat{\tilde u}|$ to the third term
\begin{eqnarray*}
\lefteqn{
\int_{-\hat H}^{\hat H} \la |\hat\grad\cdot\hat{\tilde u}||\hat\zeta_{\sigma, q_j}|^r\ra\, d\hat z}\\
&\lesssim&
\left(\int_{-\hat H}^{\hat H} \la |\hat\grad\hat{\tilde u}|^2\ra\, d\hat   z\right)^{1/2}\\
&&\mbox{}\quad\quad\times
\left(\int_{-\hat H}^{\hat H} \la|\hat\zeta_{\sigma, q_j}|^{2r/(2-r)}\ra\, d\hat z\right)^{(2-r)/(2r)}
 \left(\int_{-\hat H}^{\hat H} \la| \hat\zeta_{\sigma, q_j}|^{r}\ra\, d\hat z \right)^{(r-1)/r},
\end{eqnarray*}
the commutator estimate \eqref{AL8.1} with $p=2$ to the second term
\begin{eqnarray*}
 \lefteqn{\int_{-\hat H}^{\hat H} \la |[\hat{\tilde u}\cdot, \psi_{\sigma, q_j}\ast\phi_0\ast]\hat \grad \hat\zeta|^r\ra\, d\hat z}\\
&\lesssim&
\left(\int_{\R^d} |\grad \psi_{\sigma, q_j}\ast\phi_0||x|\, dx +\int_{\R^d}| \psi_{\sigma, q_j}\ast\phi_0|\, dx\right)^r\\
&&\mbox{}\times\left(\int_{-\hat H}^{\hat H} \la |\hat\grad\hat {\tilde u}|^2\ra\, d\hat z \right)^{r/2}\left(\int_{-\hat H}^{\hat H} \la  |\hat\zeta|^{2r/(2-r)}\ra\, d\hat z\right)^{(2-r)/2},
\end{eqnarray*}
and observing that
\begin{equation}\label{104}
\int_{\R^d} |\grad \psi_{\sigma, q_j}\ast\phi_0||x|\, dx + \int_{\R^d} | \psi_{\sigma, q_j}\ast\phi_0|\, dx\; <\; \infty,
\end{equation}
since $\psi_{\sigma, q_j}\ast\phi_0$ is a Schwartz function, the above estimate turns into
\begin{eqnarray*}
\lefteqn{\int_{-\hat H}^{\hat H} \la |\hat\zeta_{\sigma, q_j}|^r\ra\, d\hat z}\\
&\lesssim&\int_{-\hat H}^{\hat H} \la |\hat f_{\sigma, q_j}|^r\ra\, d\hat z\\
&&\mbox{} +\left(\int_{-\hat H}^{\hat H} \la |\hat\grad\hat {\tilde u}|^2\ra\, d\hat z \right)^{r/2}\left(\int_{-\hat H}^{\hat H} \la  |\hat\zeta|^{2r/(2-r)}\ra\, d\hat z\right)^{(2-r)/2}\\
&&\mbox{} +\left(\int_{-\hat H}^{\hat H} \la |\hat\grad\hat{\tilde u}|^2\ra\, d\hat   z\right)^{1/2}\\
&&\mbox{}\quad\quad\times
\left(\int_{-\hat H}^{\hat H} \la|\hat\zeta_{\sigma, q_j}|^{2r/(2-r)}\ra\, d\hat z\right)^{(2-r)/(2r)}
 \left(\int_{-\hat H}^{\hat H} \la| \hat\zeta_{\sigma, q_j}|^{r}\ra\, d\hat z \right)^{(r-1)/r}.
\end{eqnarray*}
In view of \eqref{104}, we may use the convolution estimate and Young's inequality (in the case $r>1$) to reduce this estimate to
\begin{eqnarray*}
\lefteqn{\int_{-\hat H}^{\hat H} \la |\hat\zeta_{\sigma, q_j}|^r\ra\, d\hat z}\\
&\lesssim&\int_{-\hat H}^{\hat H} \la |\hat f_{\sigma, q_j}|^r\ra\, d\hat z +\left(\int_{-\hat H}^{\hat H} \la |\hat\grad\hat {\tilde u}|^2\ra\, d\hat z \right)^{r/2}\left(\int_{-\hat H}^{\hat H} \la  |\hat\zeta|^{2r/(2-r)}\ra\, d\hat z\right)^{(2-r)/2}.
\end{eqnarray*}
By a standard covering argument, we convert the above micro-local inequality from the finite number of open balls $\{B_{\sigma}(q_j)\}_{1\le j\le J}$ to a local inequality on the annulus $\{ e^{-1}<|q|\le e\}$, so that we may replace $\hat\zeta_{\sigma, q_0} $ by $\hat\zeta_0$ and $\hat f_{\sigma, q_0}$ by $\hat f_0$. Now, we scale back to the original variables. Since \eqref{ALP2} translates into $\phi_{\ell}(x) = e^{d\ell}\phi_0(e^{\ell}x)$, it holds $\hat\zeta_0 = e^{2\ell}\zeta_{\ell}$ and $\hat f_0=f_{\ell}$, and thus
\begin{eqnarray*}
\lefteqn{e^{2r\ell}\int_{- H}^{ H} \la |\zeta_{\ell}|^r\ra\, d z}\\
&\lesssim&\int_{- H}^{ H} \la | f_{\ell}|^r\ra\, d z + \left(\int_{- H}^{ H} \la |\grad {\tilde u}|^2\ra\, d z \right)^{r/2}\left(\int_{- H}^{ H} \la  |\zeta|^{2r/(2-r)}\ra\, d z\right)^{(2-r)/2}.
\end{eqnarray*}
Since $\tilde u= \eta u$, we have
\[
\int_{-H}^H \la |\grad\tilde u|^2\ra\,dz
\;\lesssim\;
\int_{-2\delta}^{2\delta} \la |\grad u|^2\ra\,dz + \frac1{\delta^2}\int_{-2\delta}^{2\delta} \la | u|^2\ra\,dz
\;\lesssim\; 
\int_{-2\delta}^{2\delta} \la |\grad u|^2\ra\,dz.
\]
The last estimate is due to the Poincar\'e inequality which we may apply thanks to the homogeneous boundary conditions of $u$ at $z=0$. Now, \eqref{AP2.11} immediately follows from the well-know fact that
\begin{eqnarray}
\int_{-H}^H \la |\grad^s \zeta_{\ell}|^{p}\ra\, dz
&\sim&
e^{ps\ell}\int_{-H}^H \la | \zeta_{\ell}|^{p}\ra\, dz,\label{AP2.9}
\end{eqnarray}
for any $1\le p< \infty$ and any $s\in \R$, and summing over $0\le\ell\le N-1$. We display the argument for \eqref{AP2.9} for the convenience of the reader. As we will not consider applications with fractional or negative derivatives, we restrict to the case $s\in\N$. Because of \eqref{ALP1}--\eqref{ALP3}, we have $1=\sum_{\ell'=\ell-1, \ell, \ell+1} \F\phi_{\ell'}$ in the support of $\phi_{\ell}$, so that $\F \phi_{\ell} = \sum_{\ell'=\ell-1, \ell, \ell+1} \F\phi_{\ell'}\F \phi_{\ell}$. Hence $\sum_{\ell'=\ell-1, \ell, \ell+1}\phi_{\ell'}$ leaves $\phi_{\ell}$ invariant under convolution. We deduce
\[
\grad^s \zeta_{\ell}
\;=\; \grad^s \phi_{\ell} \ast \zeta
\;=\;\sum_{\ell'=\ell-1, \ell, \ell+1} \grad^s\phi_{\ell'}\ast \phi_{\ell} \ast \zeta,
\]
so that
\begin{eqnarray*}
\int_{-H}^H \la |\grad^s \zeta_{\ell}|^{p}\ra\, dz
&\lesssim& \sum_{\ell'=\ell-1, \ell, \ell+1}\int_{-H}^H \la |\grad^s\phi_{\ell'}\ast \phi_{\ell} \ast \zeta|^{p}\ra\, dz\\
&\lesssim& \sum_{\ell'=\ell-1, \ell, \ell+1} \left(\int_{\R^d}|\grad^s\phi_{\ell'}|\, dx\right)^{p}\int_{-H}^H \la | \phi_{\ell} \ast \zeta|^p \ra\, dz\\
&\stackrel{\eqref{ALP2}}{=}&\sum_{\ell'=\ell-1, \ell, \ell+1} e^{sp\ell'}\left(\int_{\R^d}|\tilde\grad^s\phi_{0}|\, d\tilde x\right)^{p}\int_{-H}^H \la |  \zeta_{\ell}|^{p}\ra\, dz\\
&\lesssim& e^{sp\ell}\int_{-H}^H \la |  \zeta_{\ell}|^{p}\ra\, dz,
\end{eqnarray*}
where we have use that
\[
\int_{\R^d}|\tilde\grad^s\phi_{0}|\, d\tilde x\;<\;\infty.
\]
For the opposite inequality, observe that because of $\zeta_{\ell} = \sum_{\ell'=\ell-1,\ell,\ell+1}\phi_{\ell'} \ast \zeta_{\ell}$, we have
\[
(\F \zeta_{\ell} )(q)
\; = \;
\sum_{\ell'=\ell-1, \ell, \ell+1} \frac{i q_j}{|q_j|^2}(\F \phi_{\ell'})(q)(\F \del_j\zeta_{\ell})(q)
\]
for any $1\le j \le d$. Thus
\[
\int_{-H}^H \la |\zeta_{\ell}|^{p}\ra\, dz
\;\lesssim\;
\left(\sum_{\ell'=\ell-1, \ell, \ell+1}\int_{\R^d}| \F^{-1}\left(\frac{i q_j}{|q_j|^2} \F \phi_{\ell'}\right)|\, dx \right)^{p}
\int_{-H}^H \la |\del_j \zeta_{\ell}|^{p}\ra\, dz.
\]
We easily compute that
\begin{eqnarray*}
\int_{\R^d} |\F^{-1} \left(\frac{iq_j}{|q_j|^2}\F\phi_{\ell}\right)|\, dx
&=& e^{(d-1)\ell} \int_{\R^d} |\F^{-1} \left(\frac{iq_j}{|q_j|^2}\F\phi_{0}\right)(e^{\ell}x)|\, dx\\
&=& e^{-\ell} \int_{\R^d} |\F^{-1} \left(\frac{iq_j}{|q_j|^2}\F\phi_{0}\right)(\tilde x)|\, d\tilde x\\
&\lesssim& e^{-\ell}.
\end{eqnarray*}
This proves \eqref{AP2.9} for $s=1$. The case $s\ge2$ follows by iteration.

\medskip

Finally, we consider the high frequency part, i.e., \eqref{AP2.4}. We show that
\begin{eqnarray}
\lefteqn{\sum_{\ell\ge N}\left(\int_{-H}^H \la |\grad^2 \zeta_{\ell}|^r\ra\, dz\right)^{1/r}}\nonumber\\
&\lesssim&\sum_{\ell\ge N}\left( \int_{-H}^H \la |f_{\ell}|^r\ra\, dz\right)^{1/r}\nonumber\\
&&\mbox{}+ e^{-N} \left(\int_{-H}^H \la|\grad u|^4\ra\, dz\right)^{1/4}\left(\int_{-H}^H \la |\grad \zeta|^{4r/(4-r)}\ra\, dz\right)^{(4-r)/4r}\label{AP2.17}.
\end{eqnarray}
Our treatment is slightly different from the one for the intermediate frequencies. Starting point is again equation \eqref{AP2.12} with $\tilde u$ replaced by $u$:
\begin{equation}\nonumber
\partial_{\hat t}\hat \zeta_{\sigma, q_j} + \hat u\cdot\hat\grad\hat\zeta_{\sigma, q_j} -\hat \laplace\hat \zeta_{\sigma, q_j}\; =\;\hat f_{\sigma, q_j} + [\hat u\cdot, \psi_{\sigma, q_j}\ast\phi_0\ast]\hat\grad \hat\zeta.
\end{equation}
Using the similar arguments as for the intermediate frequencies but with the commutator estimate \eqref{AL8.2} and $p=4$ instead of \eqref{AL8.1}, we arrive at
\begin{eqnarray*}
\lefteqn{\int_{-H}^H \la |\grad^2 \zeta_{\ell}|^r\ra\, dz}\\
&\lesssim& \int_{-H}^H \la |f_{\ell}|^r\ra\, dz\\
&&\mbox{}+ e^{-\ell r} \left(\int_{-H}^H \la|\grad u|^4\ra\, dz\right)^{r/4}\left(\int_{-H}^H \la |\grad \zeta|^{4r/(4-r)}\ra\, dz\right)^{(4-r)/4}.
\end{eqnarray*}
Notice that the analogous estimate to \eqref{AL7.1} simplifies because of \eqref{AP2.1bis}. Now, \eqref{AP2.17} follows from
\[
\sum_{\ell\ge N} e^{-\ell}
\;\lesssim\;
e^{-N}.
\]

\medskip

Compared to \eqref{AP2.0dis}, it remains to show that
\begin{eqnarray}
\lefteqn{\sum_{\ell\ge0} \left(\int_{-H}^H \la |f_{\ell}|^r\ra\, dz \right)^{1/r}}\nonumber\\
&\lesssim&
M \left(\int_{-H}^H \la |f|^r\ra\, dz \right)^{1/r}+ e^{-M} \left(\int_{-H}^H \la |\grad f|^r\ra\, dz \right)^{1/r}.\label{AP2.19}
\end{eqnarray}
We split the sum on the l.h.s.\ of \eqref{AP2.19} according to
\begin{eqnarray*}
\lefteqn{
\sum_{\ell\ge0}\left( \int_{-H}^H \la |f_{\ell}|^r\ra\, dz \right)^{1/r}}\\
&=&\sum_{0\le\ell\le M-1}\left( \int_{-H}^H \la |f_{\ell}|^r\ra\, dz \right)^{1/r}
+ \sum_{\ell\ge M }\left( \int_{-H}^H \la |f_{\ell}|^r\ra\, dz \right)^{1/r}.
\end{eqnarray*}
Thanks to the convolution estimate and
\[
\int_{\R^d} |\phi_{\ell}|\, dx\;=\; \int_{\R^d} |\phi_0|\, dx\;<\; \infty,
\]
we have for the low frequency part
\[
\sum_{0\le\ell\le M-1} \left(\int_{-H}^H \la |f_{\ell}|^r\ra\, dz \right)^{1/r}
\;\lesssim\;
M\left(\int_{-H}^H \la |f|^r\ra\, dz\right)^{1/r},
\] 
and for the high frequencies
\begin{eqnarray*}
\sum_{\ell\ge M }\left( \int_{-H}^H \la |f_{\ell}|^r\ra\, dz \right)^{1/r}
&\stackrel{\eqref{AP2.9}}{\lesssim}& \sum_{\ell\ge M }e^{-\ell} \left(\int_{-H}^H \la |\grad f_{\ell}|^r\ra\, dz\right)^{1/r} \\
&\lesssim&\sum_{\ell\ge M }e^{-\ell}\left( \int_{-H}^H \la |\grad f|^r\ra\, dz \right)^{1/r}\\
&\lesssim&
e^{-M} \left(\int_{-H}^H \la |\grad f|^r\ra\, dz\right)^{1/r}.
\end{eqnarray*}
We combine the estimates for the high and the low frequency part and obtain \eqref{AP2.19}.
\end{proof}

\section*{Acknowledgment} This work is part of my PhD thesis. I am indebted to my advisor Felix Otto for interesting and stimulating discussions on the subject of this paper. Also, I would like to thank the Max-Planck Institut f\"ur Mathematik in den Naturwissenschaften in Leipzig for support during the last year of my thesis.

\bibliography{os_lit}
\bibliographystyle{acm}
\end{document}